\documentclass[12pt]{article}

\usepackage{amssymb}
\usepackage{amsthm}
\usepackage{amsmath}
\usepackage{graphicx}
\usepackage{fullpage}
\usepackage{color}
 \numberwithin{equation}{section}

\usepackage{epsfig}\usepackage{psfrag}

\theoremstyle{plain}
\newtheorem{thm}{Theorem}[section]
\newtheorem{cor}[thm]{Corollary}

\newtheorem{lem}[thm]{Lemma}

\theoremstyle{definition}
\newtheorem{defn}[thm]{Definition}

\theoremstyle{remark}

\newtheorem{rem}[thm]{Remark}

\newcommand{\bp}{\begin{proof}[\ensuremath{\mathbf{Proof}}]}
\newcommand{\bs}{\begin{proof}[\ensuremath{\mathbf{Solution}}]}
\newcommand{\ep}{\end{proof}}

\begin{document}

\title{ Generating and Adding Flows on Locally Complete Metric Spaces}

\author{Hwa Kil Kim and Nader Masmoudi}
\maketitle

\begin{abstract}
As a generalization of a vector field on a manifold, the notion of an arc field on a locally complete metric space was introduced in \cite{BC}.
In that  paper, the  authors proved an analogue of the  Cauchy-Lipschitz Theorem i.e they showed the existence and uniqueness of solution curves for  a time
independent arc field. In this paper, we extend the result to the
 time dependent case, namely
 we show the existence and uniqueness of solution
curves for  a time dependent arc field. We also introduce the
notion of the  sum of two time dependent arc fields
and show  existence and uniqueness of solution curves for this sum.
\end{abstract}


\section{Introduction}

Vector fields play an important role on manifolds.
In particular  they allow the  study of dynamics
  on  the manifold.  On metric spaces and
in the absence of a differential
structure, the notion of  arc fields was introduced in
\cite{BC}.  Under some regularity assumption, the authors of \cite{BC} proved
the existence of  solution curves for  a time
independent arc field. Their result can be seen as an
extension of the  Cauchy-Lipschitz Theorem. The goal of this paper
is to define a notion of sum of two  arc fields
and construct a  unique  solution curve  for this sum.
We also generalize \cite{BC} to the time dependent case.

Let us also mention that
the generalization of the notion of differential equations from manifolds to metric spaces is a natural question.
 In this direction,  there are many other  approaches which can be found in  \cite{A}, \cite{N},
\cite{P1}, \cite{P2}, \cite{Penot07} and \cite{CG09}. A basic idea that all approaches have in
common is to replace the concept of a vector field by a suitable family of curves (herein called an 'arc field' following \cite{BC})
 each of which supplies the direction of travel at the point from which it issues. We borrow the idea of \cite{BC} which shows the existence
 of flows corresponding time independent arc fields on locally complete metric spaces whereas all others have predominantly assumed that the underlying
 metric space is locally compact.

Let us now explain our  motivation behind this work.
In   \cite{KM}, we  study systems coupling fluids and polymers.
In it most generality the phase space for the polymers is given by a  metric space (see \cite{CZ10}).
When  the phase space of the  polymers is a manifold,
 we get a system coupling the Navier-Stokes equation for the fluid velocity with a Fokker-Planck equation describing the evolution of the
polymer density (see for instance \cite{Constantin05,CM08}).
The coupling comes from an extra stress term in the fluid equation due to the polymers. There is also a drift term in the Fokker-Planck
 equation that depends on the spatial gradient of fluid velocity. It can be seen that the Fokker-Planck equation
 has a flow structure on the set of
 probability densities of polymers. More specifically, let $\mathcal{M}$ be the set of all Borel probability
measures defined on the manifold (phase space of polymers) then we can put
 a metric structure on $\mathcal{M}$ using  the Wasserstein distance.
Once $\mathcal{M}$ is equipped with the Wasserstein distance, the Fokker-Planck equation can be considered as
 the  sum of two flows on $\mathcal{M}$.
One is the gradient flow corresponding to the entropy functional on $\mathcal{M}$ and the other one is a
drift term which is  generated by the spatial gradient of
the fluid velocity which depends on time. If the phase space of polymers is not a manifold but just a metric space then we don't have Fokker-Planck
equation
any more.
But the flow interpretation is still available to describe the evolution of the
 polymer density if we know how to generate and add flows on metric spaces.  Achieving this
is one of the goal of this paper.


We briefly summarize the contents of each section. In section 2, we study time dependent arc fields,
solution curves, and sufficient conditions under which we can prove the existence of solution curves for  arc
fields. We also show the continuous dependence of solutions on initial conditions from which we can get
the uniqueness of the solution curve.
In section  3, we introduce the notion of
 solution curve for the  sum of  two arc fields. By imposing a kind of commutation  law
on two time dependent arc fields, we prove the existence of solution curves. We also get the uniqueness of
 a  solution curve to the  sum of two arc field by
showing the continuous dependence of solution curves on the initial conditions.

\section{Generating Flows}

\subsection{Time dependent arc fields}
Let $X$ be a locally complete metric space with a  metric $d$.
\begin{defn}
A time dependent arc field on $X$ is a family of maps $\Phi( \cdot ; \cdot,\cdot): [0,\infty) \times   X\times[0,1]\rightarrow X$
such that for all $t\in[0,\infty), x\in X, $ we have   $ \Phi(t; x,0)=x,$
$$\rho(x,t):=\sup_{h\neq k}\frac{d(\Phi(t;x,h),\Phi(t;x,k))}{|h-k|} <\infty $$
and the function $\rho(x,t)$ is locally bounded, namely  for all $t,x \in [0,\infty) \times X $,
there exist $r,l > 0$ such that
$$\rho(x,t;r,l):=\sup_{y\in B(x,r),|t-s|\leq l }\{\rho(y,s) \}<\infty.  $$
\end{defn}
One can interpret $\Phi(t;b,\cdot) : [0,1] \rightarrow X$ as a curve on $X$  starting from $b$
 to $\Phi(t;b,1)$. This  gives the direction of the  curve in some
sense. Notice that, for fixed $b\in X$,  the direction  given by $\Phi(t;b,.)$ depends on the time $t$.
 Besides,  $\rho(x,t)$ can be understood as the upper bound on the
speed of the curve $\Phi(t;b,\cdot)$. For the convenience, we will
use the notation $\Phi^t_h(b):=\Phi(t;b,h).$
\begin{defn}
For given $a\in X$ and $t\in [0,\infty),$ a solution curve of $\Phi$ with initial position $a$ at time $t$ is a map
$\sigma:[t, t + c) \rightarrow X$ (for some $c > 0$)
such that $\sigma(t)=a$ and for each $s\in[t,t+c)$
\begin{equation}\label{Definition-solution}
\lim_{h\rightarrow 0^+}\frac{d(\sigma(s+h),\Phi^{s}_h(\sigma(s)))}{h}=0
\end{equation}
\end{defn}

We introduce some conditions on the  time dependent arc field $\Phi$. Motivations for Condition A and B were
 already given in \cite{BC}.
Condition C is about the time regularity of $\Phi$.

{\bf Condition A:} There is a function $\Lambda : X\times X \times [0,1] \rightarrow (-\infty,\infty)$ such that for each $a\in X$ and $t\in [0,\infty)$,
  there are constants $r_a>0,$
$\epsilon_a\in(0,1]$ and $T_t>t$ such that $\Lambda $ is bounded above on $  B(a,r_a)\times
B(a,r_a)\times [0,\epsilon_a]$  and
\begin{equation}\label{Condition-A}
d(\Phi^s_h(a_1),\Phi^s_h(a_2)) \leq d(a_1,a_2)(1+h\Lambda(a_1,a_2,h))
\end{equation}
for all $a_1,a_2\in B(a,r_a),$ $h\in[0,\epsilon_a]$ and $s\in[t,T_t]$.

{\bf Condition B:} There is a function $\Omega:X\times[0,1]\times[0,1]\rightarrow [0,\infty)$ such that for each $a\in X$ and $t\in [0,\infty)$,
there are constants $r_a>0$, $\epsilon_a\in(0,1]$ and  $T_t>t$ for which $\Omega$
  is bounded on $ B(a,r_a)\times[0,\epsilon_a]\times[0,\epsilon_a]$ and
\begin{equation}\label{Condition-B1}
d(\Phi^s_{l+h}(b),\Phi^s_h\circ \Phi_l^s(b)) \leq hg(l,h)\Omega(b,l,h)
\end{equation}
for all $b\in B(a,r_a)$ , $l,h\in[0,\epsilon_a]$ and $s\in[t,T_t]$ where $g:[0,\epsilon_a]\times [0,\epsilon_a]\rightarrow [0,\infty)$ satisfies
\begin{equation}\label{Condition-B2}
\lim_{l,h\rightarrow0^+} g(l,h)=0 \quad {\rm and} \quad \sum_{i\in Z^+, 2^{-i}\leq \epsilon_a}g(2^{-i},2^{-i})<\infty
\end{equation}

{\bf Condition C:} For each $a\in X$ and $t\in [0,\infty)$, there are constants $r_a>0$, $\epsilon_a\in(0,1]$, $T_t>t$, $0<\alpha<1$ and $C>0,$
such that
\begin{equation}\label{Condition-C}
d(\Phi^{s_1}_h(b),\Phi^{s_2}_h(b)) \leq  C h |s_1-s_2|^\alpha
\end{equation}
for all $b\in B(a,r_a)$ , $h\in[0,\epsilon_a]$ and $s_1,s_2\in[t,T_t]$.

\begin{rem}
Once we have fixed $a\in X$, $t\in [0,\infty)$ and fixed constants $r_a,\epsilon_a, T_t$ then functions $\Lambda$ and $\Omega$ are bounded above.
We denote upper bound of $\Lambda$(respectively $\Omega$) by $K_A$($K_B$).
\end{rem}

As a simple observation, by combining Condition B and C, if $b, \Phi^s_l(b)\in B(a,r_a)$ then we have
\begin{align}\label{Condition-D}\nonumber
d(\Phi_{l+h}^s(b), \Phi_h^{s+l}\circ\Phi_l^s(b))&\leq d(\Phi_{l+h}^s(b), \Phi_h^s\circ\Phi_l^s(b)) +
d(\Phi_h^s\circ\Phi_l^s(b), \Phi_h^{s+l}\circ\Phi_l^s(b))\\\nonumber
   &\leq hg(l,h)K_B +Chl^\alpha\\
   & = h(g(l,h)K_B+ Cl^\alpha)=: h\tilde{g}(l,h)
\end{align}

\begin{lem}\label{One-step}
For a  given $a\in X$ and $t\in[0,\infty)$, let $r_a$, $\epsilon_a$ and $T_t$ be  the
constants in Condition A, B and C. If $b_1,b_2\in B(a,r_a)$,
$2h\in[0,\epsilon_a]$, $s,s+h\in[t,T_t]$ and $\Phi^s_h(b_1),\Phi^s_h(b_2) \in  B(a,r_a)$ then we have
\begin{equation}\label{Eq:lemma:One-step}
d(\Phi_h^{s+h}\circ\Phi_h^s(b_1),\Phi_{2h}^s(b_2)) \leq d(b_1,b_2)(1+hK_A)^2+h\tilde{g}(s,h)
\end{equation}
\end{lem}
\begin{proof}

\begin{figure}[h]
\centering\centerline{
\includegraphics[width=15cm, height=10cm]{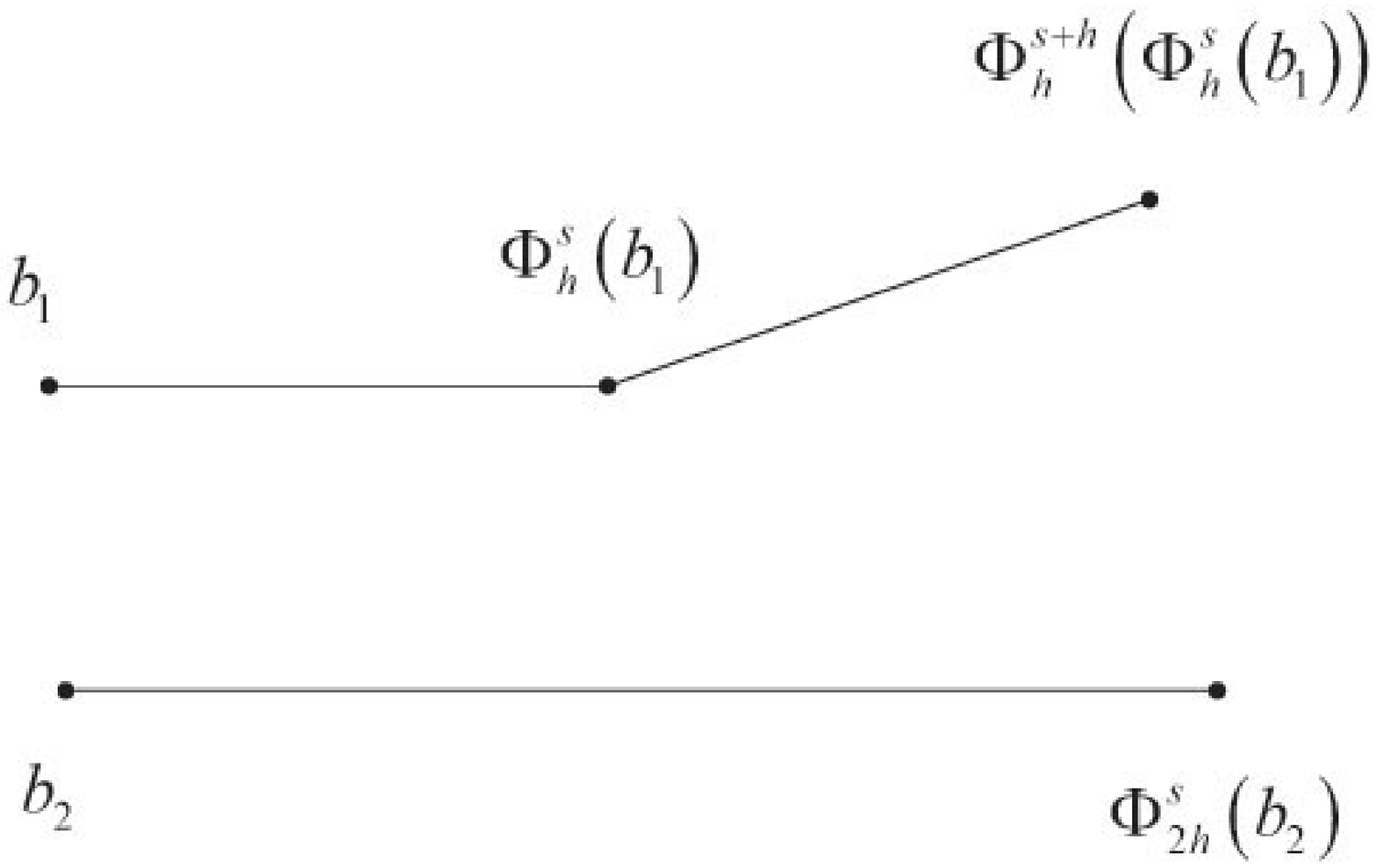}}
    \caption{
    }
\label{fig6_overload}
\end{figure}

Triangle inequality gives
\begin{align}\label{Eq1:proof:lemma:One-step}
d(\Phi_h^{s+h}\circ\Phi_h^s(b_1),\Phi_{2h}^s(b_2)) \leq d(\Phi_h^{s+h}\circ\Phi_h^s(b_1),\Phi_h^{s+h}\circ\Phi_h^s(b_2))
+ d(\Phi_{2h}^{s}(b_2),\Phi_h^{s+h}\circ\Phi_h^s(b_2))
\end{align}
\\
For the first term in the right hand side of (\ref{Eq1:proof:lemma:One-step}), we use Condition A twice
\begin{equation}\label{Eq1}
d(\Phi_h^{s+h}\circ\Phi_h^s(b_1),\Phi_h^{s+h}\circ\Phi_h^s(b_2))\leq
d(\Phi_h^{s}(b_1),\Phi_h^s(b_2))(1+hK_A)\leq d(b_1,b_2)(1+hK_A)^2
\end{equation}
\\
For the second term, we exploit (\ref{Condition-D}) to get
\begin{equation}\label{Eq2}
d(\Phi_{2h}^{s}(b_2),\Phi_h^{s+h}\circ\Phi_h^s(b_2))\leq h\tilde{g}(s,h)
\end{equation}
\\
We combine (\ref{Eq1:proof:lemma:One-step}), (\ref{Eq1}) and (\ref{Eq2}) to finish the proof.
\end{proof}

\begin{rem}\label{Remark:One-step}  In general, we have
\begin{equation}\label{Eq:Remark:One-step}
d(\Phi_{h_2}^{s+h_1}\circ\Phi_{h_1}^s(b_1),\Phi_{h_1+h_2}^s(b_2)) \leq d(b_1,b_2)(1+h_1K_A)(1+h_2K_A)+h_2\tilde{g}(h_1,h_2)
\end{equation}

\end{rem}


\begin{lem}\label{Lemma1}
For a  given $a\in X$ and $t\in[0,\infty)$, let $r_a$, $\epsilon_a$ and $T_t$ be  the
constants in Condition A, B and C. If $b_1,b_2\in B(a,r_a)$,
$h,l+h\in[0,\epsilon_a]$, $s,s+l\in[t,T_t]$ and $\Phi^s_l(b_2) \in  B(a,r_a)$ then we have
$$d(\Phi_h^{s+l}(b_1),\Phi_{l+h}^s(b_2))\leq d(b_1,\Phi_l^s(b_2)) +h\eta_s(b_1,b_2,l,h)$$
where $\eta_s(b_1,b_2,l,h):= d(b_1,\Phi_l^s(b_2))K_A+\tilde{g}(l,h) $. Furthermore, we notice that $\eta_s(b_1,b_1,l,h) $ converges to $0$ as
$d(b_1,b_2), l,h\rightarrow 0.$
\end{lem}
\begin{proof}
We use Condition A and (\ref{Condition-D}) to get
\begin{align*}
d(\Phi_h^{s+l}(b_1),\Phi_{l+h}^s(b_2)) &\leq d(\Phi_h^{s+l}(b_1),\Phi_h^{s+l}\circ\Phi_{l}^s(b_2))+
             d(\Phi_h^{s+l}\circ\Phi_{l}^s(b_2),\Phi_{l+h}^s(b_2))\\
             & \leq d(b_1,\Phi_l^s(b_2))(1+hK_A) + h\tilde{g}(l,h)\\
             & = d(b_1,\Phi_l^s(b_2)) +h[ d(b_1,\Phi_l^s(b_2)K_A+\tilde{g}(l,h) ]
\end{align*}

and trivially $\eta_s(b_1,b_1,l,h)\rightarrow 0$ as $d(b_1,b_2), l,h\rightarrow 0$.
\end{proof}

\begin{lem}\label{New}
For  a  given $a\in X$ and $t\in[0,\infty)$, let $r_a$, $\epsilon_a$ and $T_t$ be the
 constants in Condition A, B and C.
If $b_1,b_2\in B(a,r_a)$ and $h \in[0,\epsilon_a]$, $s,u\in[t,T_t]$  then we have,
$$d(\Phi_h^s(b_1),\Phi_{h}^u(b_2))\leq d(b_1,b_2)(1+hK_A) + Ch|s-u|^\alpha$$
\end{lem}
\begin{proof}
We combine Condition A and C to get
\begin{align*}
d(\Phi_h^s(b_1),\Phi_{h}^u(b_2))& \leq d(\Phi_h^s(b_1),\Phi_{h}^s(b_2)) + d(\Phi_h^s(b_2),\Phi_{h}^u(b_2))\\
& \leq d(b_1,b_2)(1+hK_A) + Ch|s-u|^\alpha
\end{align*}

\end{proof}

\begin{lem}\label{Check-solution1}
For a  given $a\in X$ and $t\in[0,\infty)$, let $r_a$, $\epsilon_a$ and $T_t$ be the
 constants in Condition A, B and C.
Assume $b_1,b_2\in B(a,r_a)$ and  $s,s+h\in[t,T_t]$. Define a polygonal path $p(l):[0,h]\rightarrow X$ starting at $b_1\in X$ as follows; $p(l):=\Phi^s_{l}(b_1)$ for
$0\leq l\leq r_1$ and  $p(l):= \Phi^{s+r_i}_{l-r_i}(p(r_i))$ for
 $r_i\leq l\leq r_{i+1}$ with $0\leq r_1\leq \dots\leq r_i\leq r_{i+1}\leq \dots\leq r_k=h.$ Then we have
\begin{equation}
d(p(l), \Phi_l^s(b_2)) \leq d(b_1,b_2) + h \max_{1\leq i \leq k} \eta_s(p(r_i),b_2,r_i,r_{i+1}-r_i)
\end{equation}
for $0\leq l\leq h$. Furthermore, we have
\begin{equation}\label{Eq1:lemma:Check-solution}
\lim_{h\rightarrow 0}\max_{0\leq r_i \leq h} \eta_s(p(r_i),b_2,r_i,r_{i+1}-r_i) = 0
\end{equation}
\end{lem}
\begin{proof}
If $0\leq l\leq r_1$ then by Lemma \ref{Lemma1}, we have
$$d(p(l), \Phi_{l}^s(b_2))\leq d(b_1,b_2) + l\eta_s(b_1,b_2,0,l)  $$
For $r_1\leq l\leq r_2$, we use Lemma \ref{Lemma1} twice to get
\begin{align*}
d(p(l),\Phi_{l}^s(b_2))&= d(\Phi^{s+r_1}_{l-r_1}(p(r_1))   ,\Phi^s_{r_1+l-r_1}(b_2)) \\
& \leq d(p(r_1),\Phi_{r_1}^s(b_2)) +(l-r_1)\eta_s(p(r_1),b_2,r_1,l-r_1)\\
& \leq d(b_1,b_2) + r_1\eta_s(b_1,b_2,0,r_1) + (l-r_1)\eta_s(p(r_1),b_2,r_1,l-r_1)
\end{align*}
In general, for $r_i\leq l\leq r_{i+1}$, we have
\begin{align*}
d(p(l),\Phi_{l}^s(b_2))& \leq  d(b_1,b_2) + r_1\eta_s(b_1,b_2,0,r_1) + (r_2-r_1)\eta_s(p(r_1),b_2,r_1,r_2-r_1) \\
           & \qquad \qquad \cdots \quad +(l-r_i)\eta_s(p(r_i),b_2,r_i,l-r_i)
\end{align*}
which gives
$$d(p(l), \Phi_l^s(b_2)) \leq d(b_1,b_2) + h \max_{1\leq i \leq k} \eta_s(p(r_i),b_2,r_i,r_{i+1}-r_i)$$

Equation (\ref{Eq1:lemma:Check-solution}) is almost trivial since $d(p(r_i),b_2),r_i,r_{i+1}-r_i $ converge to $0$ as $h \rightarrow 0.$
\end{proof}


\subsection{Existence and uniqueness of a solution curve}
The proof of the  next theorem is similar to the one in \cite{BC}.  We can also
 think of it as a corollary of Theorem \ref{Theorem:existence:sum}. But, to give an
idea for the proof of Theorem \ref{Theorem:existence:sum} which is more complicated, we give a full proof here.
\begin{thm}[Existence]\label{Theorem1}
Let $\Phi:X\times[0,1]\times[0,\infty)$ be an arc field satisfying Condition A, B and C. For  a given $a\in X$ and $t\in [0,\infty)$,
 there exists  a solution curve $\sigma:[t,t+c)\rightarrow  X$ with initial position $ a$ at time $t.$
\end{thm}
\begin{proof}
For a positive integer n, we define the n-th discretized solution by
\begin{align*}
\xi_n(s):=
\left \{ \begin{matrix} \Phi_s^t(a) & 0\leq s\leq \frac{1}{2^n}\\
\\
\Phi_{s-2^{-n}}^{t+2^{-n}}(\xi_n(\frac{1}{2^n})) & \frac{1}{2^n}\leq s \leq  \frac{2}{2^n}\\
.\\
.\\
.\\
\Phi_{s-i\cdot 2^{-n}}^{t+i\cdot 2^{-n}}(\xi_n(\frac{i}{2^n})) & \frac{i}{2^n}\leq s\leq \frac{i+1}{2^n}\\
.\\
.
\end{matrix} \right.
\end{align*}
\\
Suppose  $r,l>0$ are chosen so that $\rho(a,t;r,l)<\infty$. If $\rho(a,t;r,l)=0,$ then $\sigma(s):=a$ defines a solution curve. Thus we assume
$\rho(a,t;r,l)>0,$ and let
\begin{equation}
c:= \min\left \{\frac{r}{\rho(a,t;r,l)},l \right \}.
\end{equation}
It is easy to see that we have $\xi_n(s)\in B(a,r)$ for $0\leq s<
c$. This implies $\{\xi_n\}_{n=1}^\infty$ is equi-Lipschitz with Lipschitz constant $\rho(a,t;r,l).$ Moreover, by choosing $r$ smaller if necessary, we may assume
that there are constants $K_A$, $K_B$ and $\epsilon\in(0,1]$ such
that $\Lambda(p,q,h)\leq K_A$ and $\Omega(p,l',h)\leq K_B$ for all
$p,q\in B(a,r)$ and $l',h\in [0,\epsilon]$. We may also
assume that
 $\overline{B(a,r)}$ is a complete metric space.



\begin{figure}[h]
\centering\centerline{
\includegraphics[width=16cm, height=10cm]{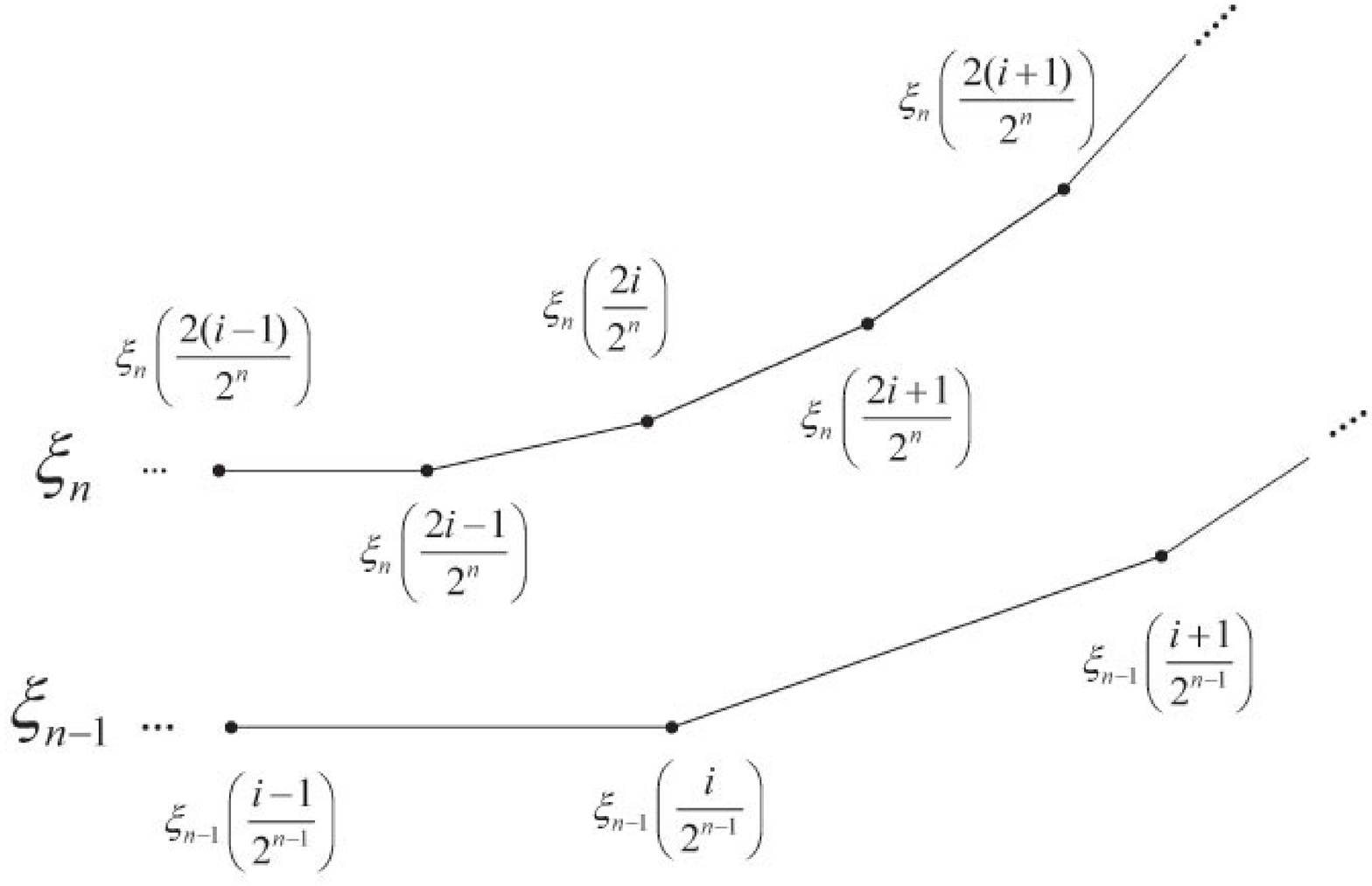}}
    \caption{
    }
\label{fig1_overload}
\end{figure}

Let us first estimate the uniform distance between $\xi_n$ and $\xi_{n-1}$.
We apply Lemma \ref{One-step} with $h=1/2^n$ and $b_1=b_2= a$ to get
\begin{align*}
d\bigl (\xi_n(\frac{2}{2^n}),\xi_{n-1}(\frac{2}{2^n}) \bigr)&\leq \frac{1}{2^n}\tilde{g}\bigl (\frac{1}{2^n},\frac{1}{2^n} \bigr)
\end{align*}
Similarly, we apply Lemma \ref{One-step} multiple times and get
\begin{align*}
d \bigl (\xi_n(\frac{2\cdot 2 }{2^n}),\xi_{n-1}(\frac{2\cdot 2}{2^n}) \bigr)&\leq d \bigl(\xi_n(\frac{2}{2^n}),\xi_{n-1}(\frac{2}{2^n})\bigr )
\bigl(1+\frac{K_A}{2^n} \bigr)^2 + \frac{1}{2^n}\tilde{g}\bigl(\frac{1}{2^n},\frac{1}{2^n}\bigr)\\
& \leq \frac{1}{2^n}\tilde{g}\bigl(\frac{1}{2^n},\frac{1}{2^n}\bigr) \bigl[1+(1+\frac{K_A}{2^n})^2\bigr]\\
d\bigl(\xi_n(\frac{3\cdot 2 }{2^n}),\xi_{n-1}(\frac{3\cdot 2}{2^n})\bigr)&\leq \frac{1}{2^n}\tilde{g}\bigl(\frac{1}{2^n},\frac{1}{2^n}\bigr)
\bigl[1+(1+\frac{K_A}{2^n})^2+ (1+\frac{K_A}{2^n})^{2\cdot 2}\bigr]
\end{align*}
In general, for all i so that $ i\cdot 2/2^{-n}\leq c $, we have
\begin{align}\label{Eq1:Theorem:existence}\nonumber
d\bigl(\xi_n(\frac{i\cdot 2 }{2^n}),\xi_{n-1}(\frac{i\cdot 2}{2^n})\bigr)&\leq \frac{1}{2^n}\tilde{g}\bigl(\frac{1}{2^n},\frac{1}{2^n}\bigr)\sum_{j=0}^{i-1}
(1+2^{-n}K_A)^{2\cdot j}\\\nonumber
& =\frac{1}{2^n}\tilde{g}\bigl(\frac{1}{2^n},\frac{1}{2^n}\bigr)\frac{(1+2^{-n}K_A)^{2i}-1}{(1+2^{-n}K_A)^2-1}\\\nonumber
&\leq \tilde{g}\bigl(\frac{1}{2^n},\frac{1}{2^n}\bigr) \frac{(1+2^{-n}K_A)^{c2^n}-1}{K_A(2+2^{-n}K_A)}\\\nonumber
&\leq \tilde{g}\bigl(\frac{1}{2^n},\frac{1}{2^n}\bigr) \frac{e^{cK_A}-1}{K_A(2+2^{-n}K_A)}\\
&= \tilde{g}\bigl(\frac{1}{2^n},\frac{1}{2^n}\bigr) K
\end{align}
where $K:=\frac{e^{cK_A}-1}{2K_A}$ is a constant independent of $n$.

So for any $s\in [0,c)$, let $i$ be an integer such that
$$\frac{2i}{2^n}\leq s < \frac{2(i+1)}{2^n}$$
then we have
\begin{align}\label{DiscriteOnestep-difference}\nonumber
d(\xi_n(s),\xi_{n-1}(s)) &\leq d\bigl(\xi_n(s),\xi_n(\frac{2i}{2^n})\bigr) + d\bigl(\xi_n(\frac{2i}{2^n}),\xi_{n-1}(\frac{2i}{2^n})\bigr)
+ d\bigl(\xi_{n-1}(\frac{2i}{2^n}),\xi_{n-1}(s)\bigr)\\\nonumber
& \leq  {\rm Lip}(\xi_n)\bigl(s-\frac{2i}{2^n}\bigr) + \tilde{g}\bigl(\frac{1}{2^n},\frac{1}{2^n}\bigr) K + {\rm Lip}(\xi_{n-1})\bigl(s-\frac{2i}{2^n}\bigr)\\
 & \leq \frac{4}{2^n}\rho(a,t;r,l) + \tilde{g}\bigl(\frac{1}{2^n},\frac{1}{2^n}\bigr)K
\end{align}
where we exploit the equi-Lipschitz property of $\xi_n$ and (\ref{Eq1:Theorem:existence}).


Next, we exploit (\ref{DiscriteOnestep-difference}) to show that  $\{\xi_n \}$ is a Cauchy sequence in the uniform topology. For any $s\in [0,c)$, we have
\begin{align}\label{Discrite-difference}\nonumber
d(\xi_n(s),\xi_{n+m}(s)) &\leq \sum_{j=0}^{m-1}d(\xi_{n+j}(s),\xi_{n+j+1}(s)) \\\nonumber
 & \leq \sum_{j=0}^{m-1} \frac{4}{2^{(n+j+1)}}\rho(a,t;r,l) + \tilde{g}\bigl(\frac{1}{2^{(n+j+1)}},\frac{1}{2^{(n+j+1)}}\bigr)K\\
 & \leq \frac{4}{2^{n+1}}\rho(a,t;r,l) + K\sum_{j=n+1}^\infty \tilde{g}(2^{-j},2^{-j}) \rightarrow 0 \quad {\rm as }\quad n\rightarrow \infty
\end{align}

Since $\xi_n(s)$ is in the complete space $\overline{B(a,r)},$ we know that $\xi_n$ converges uniformly and we define
$\tilde{\sigma}:[0,c)\rightarrow X$ by
$$\tilde{\sigma}(s)=\lim_{n\rightarrow \infty}\xi_n(s) $$

It is trivial to see $\tilde{\sigma}(0)=a$ and let us check that

\begin{equation}\label{Eq-solution1}
\lim_{h\rightarrow 0} \frac{d(\tilde{\sigma}(s+h), \Phi_h^{t+s}(\tilde{\sigma}(s)))}{h} = 0
\end{equation}
holds for all $s\in[0,c)$.
Let $\epsilon>0$ and $h>0$  be fixed such that $s+h<c.$
From triangle inequality, we have
\begin{align}\label{Eq1:Check-solution1}\nonumber
\frac{d(\tilde{\sigma}(s+h), \Phi_h^{t+s}(\tilde{\sigma}(s)))}{h}  &\leq \frac{d(\tilde{\sigma}(s+h), \xi_n(s+h))}{h}
          + \frac{d(\xi_n(s+h), \Phi_h^{t+s}(\xi_n(s)))}{h} \\
            & \qquad \qquad \qquad + \frac{d(\Phi_h^{t+s}(\xi_n(s)) \Phi_h^{t+s}(\tilde{\sigma}(s)))}{h}
\end{align}
Since $\xi_n$ converges uniformly, we can choose $n$ large enough so that
\begin{equation}\label{Eq2:Check-solution1}
 \frac{d(\tilde{\sigma}(s+h), \xi_n(s+h))}{h} + \frac{d(\Phi_h^{t+s}(\xi_n(s)) \Phi_h^{t+s}(\tilde{\sigma}(s)))}{h}\leq \frac{\epsilon}{2}
 \end{equation}
We combine (\ref{Eq1:Check-solution1}) and (\ref{Eq2:Check-solution1}) to get
\begin{equation}\label{Eq3:Check-solution1}
 \frac{d(\tilde{\sigma}(s+h), \Phi_h^{t+s}(\tilde{\sigma}(s)))}{h}
\leq  \frac{d(\xi_n(s+h), \Phi_h^{t+s}(\xi_n(s)))}{h} +\frac{\epsilon}{2}
\end{equation}
We need to estimate the second term of (\ref{Eq3:Check-solution1}).

Let $i$ be such that $i/2^n\leq s<(i+1)/2^n,$ then
\begin{align}\label{Eq3}\nonumber
 d(\xi_n(s+h), \Phi_h^{t+s}(\xi_n(s)))&\leq  d\bigl(\xi_n(s+h), \xi_n(\frac{i}{2^n}+h)\bigr) +  d\bigl(\xi_n(\frac{i}{2^n}+h), \Phi_h^{t+\frac{i}{2^n}}
 (\xi_n(\frac{i}{2^n}))\bigr)\\
& \qquad \qquad +  d\bigl(\Phi_h^{t+\frac{i}{2^n}}(\xi_n(\frac{i}{2^n})), \Phi_h^{t+s}(\xi_n(s))\bigr)
 \end{align}
 Let us estimate the righthand side of (\ref{Eq3}) term by term.  First, by the Lipschitz property of $\xi_n$, we have
\begin{align}\label{Eq4}\nonumber
d\bigl(\xi_n(s+h), \xi_n(\frac{i}{2^n}+h)\bigr) &\leq \bigl(s-\frac{i}{2^n}\bigr)\rho(a,t;r_a,\epsilon_a)\\
&\leq \frac{1}{2^n}\rho(a,t;r_a,\epsilon_a)
 \end{align}
For the second term, we use Lemma \ref{Check-solution1} with $\xi_n(\frac{i}{2^n}):=b_1=b_2$.
\begin{align}\label{Eq5}
d\bigl(\xi_n(\frac{i}{2^n}+h), \Phi_h^{t+\frac{i}{2^n}}(\xi_n(\frac{i}{2^n}))\bigr)\leq h \sup_{0\leq r\leq h} \eta\bigl(\xi_n(\frac{i}{2^n} + r),
\xi_n(\frac{i}{2^n}),r,\frac{1}{2^n}\bigr)
 \end{align}
By using Lemma \ref{New} and the Lipschitz property of $\xi_n$, we can estimate the last term
 \begin{align}\label{Eq6}\nonumber
d\bigl(\Phi_h^{t+\frac{i}{2^n}}(\xi_n(t+\frac{i}{2^n})), \Phi_h^s(\xi_n(s))\bigr)&\leq d\bigl(\xi_n(\frac{i}{2^n}),\xi_n(s)\bigr)(1+hK_A)+\frac{Ch}{2^{n\alpha}}\\
& \leq \frac{1}{2^n}\rho(a,t;r_a,\epsilon_a)(1+hK_A) + \frac{Ch}{2^{n\alpha}}
 \end{align}
We combine equations (\ref{Eq3}),(\ref{Eq4}),(\ref{Eq5}) and (\ref{Eq6}), and assume $n$ is large enough to have
\begin{align}\label{Eq7}\nonumber
 d(\xi_n(s+h), \Phi_h^{t+s}(\xi_n(s)))&\leq  h \sup_{0\leq r\leq h} \eta\bigl(\xi_n(\frac{i}{2^n} + r),
\xi_n(\frac{i}{2^n}),r,\frac{1}{2^n}\bigr) +\frac{1}{2^n}\rho(a,t;r_a,\epsilon_a)(2+hK_A) + \frac{Ch}{2^{n\alpha}}\\
&\leq h \sup_{0\leq r\leq h} \eta \bigl(\xi_n(\frac{i}{2^n} + r),
\xi_n(\frac{i}{2^n}),r,\frac{1}{2^n}\bigr)+ \frac{\epsilon}{2}
\end{align}
We combine (\ref{Eq3:Check-solution1}) and (\ref{Eq7}), and let $h\rightarrow 0$ with $n$  large
\begin{align*}
\lim_{h \rightarrow 0} \frac{d(\tilde{\sigma}(s+h), \Phi_h^{t+s}(\tilde{\sigma}(s)))}{h}
&\leq  \lim_{h \rightarrow 0}  \sup_{0\leq r\leq h} \eta\bigl(\xi_n(\frac{i}{2^n} + r),
\xi_n(\frac{i}{2^n}),r,\frac{1}{2^n}\bigr) +\epsilon\\
& =\epsilon
\end{align*}
This gives (\ref{Eq-solution1}). We define $\sigma:[t,t+c )\rightarrow X$ by
$$ \sigma(s):=\tilde{\sigma}(s-t),\qquad s\in[t,t+c)$$
It is trivial that  $\sigma(t)=\tilde{\sigma}(0)=a$ and (\ref{Eq-solution1}) implies that $\sigma$ satisfies (\ref{Definition-solution}). This  concludes  the proof.
\end{proof}

\begin{thm}[Uniqueness]\label{Theorem-uniqueness1}
Let $\sigma_{a,t}:[t,t+c)\rightarrow X$ be a solution curve of an arc field $\Phi$ with initial position $a$ at time $t$, and
let $\sigma_{b,u}:[u,u+c)\rightarrow X$ be a solution curve of $\Phi$ with initial position $b$ at time $u$.
 Then we have
$$d(\sigma_{a,t}(t+s), \sigma_{b,u}(u+s))\leq e^{K_As}d(a,b) + \tilde{C}|t-u|^\alpha , \quad {\rm for} \quad s\in[0,c) $$
where $\tilde{C}$ is a constant depending only on $c$ and $K_A$.
\end{thm}

\begin{proof}
Let us first define $\tilde{\sigma}_{\alpha,\tau} : [0,c)\rightarrow X$ by
$$ \tilde{\sigma}_{\alpha,\tau}(s):={\sigma}_{\alpha,\tau}(\tau+s),\qquad {\rm for} \quad s\in[0,c) $$
where $\alpha \in X$ and $\tau\in[0,\infty)$. If we can show
\begin{equation}\label{Check:uniqueness}
d(\tilde{\sigma}_{a,t}(s), \tilde{\sigma}_{b,u}(s))\leq e^{K_As}d(a,b) + \tilde{C}|t-u|^\alpha , \quad {\rm for} \quad s\in[0,c)
\end{equation}
for some constant $\tilde{C}$ which is depending only on $c$ and $K_A$ then we are done.

Triangle inequality gives
\begin{align}\label{Eq1:Theorem:uniqueness1}
d(\tilde{\sigma}_{a,t}(s), \tilde{\sigma}_{b,u}(s)) \leq d(\tilde{\sigma}_{a,t}(s), \tilde{\sigma}_{b,t}(s)) + d( \tilde{\sigma}_{b,t}(s),
\tilde{\sigma}_{b,u}(s))
\end{align}
\\
First, let us estimate the second term in the right hand side of (\ref{Eq1:Theorem:uniqueness1}).  Let $\xi_n^t$ be the n-th discretized solution of
$\tilde{\sigma}_{b,t}$ and let
$\xi_n^u$ be the n-th discretized solution of $\tilde{\sigma}_{b,u}$. We exploit Condition C and get
\begin{align}
d_1:= d \bigl( \xi_n^t(\frac{1}{2^n}), \xi_n^u(\frac{1}{2^n})\bigr) \leq \frac{C}{2^n}|t-u|^\alpha
\end{align}
\\
Again,  triangle inequality gives
\begin{align}\label{Eq2:Theorem:uniqueness1}\nonumber
d_2:&= d\bigl( \xi_n^t(\frac{2}{2^n}), \xi_n^u(\frac{2}{2^n})\bigr) \\\nonumber
&\leq d\bigl( \xi_n^t(\frac{2}{2^n}), \Phi_{2^{-n}}^{u+2^{-n}}(\xi_n^t(\frac{1}{2^n}))\bigr)
 +d\bigl( \Phi_{2^{-n}}^{u+2^{-n}}(\xi_n^t(\frac{1}{2^n})), \xi_n^u(\frac{2}{2^n})\bigr)\\
 &= d\bigl( \Phi_{2^{-n}}^{t+2^{-n}}(\xi_n^t(\frac{1}{2^n})), \Phi_{2^{-n}}^{u+2^{-n}}(\xi_n^t(\frac{1}{2^n}))\bigr)
 +d \bigl( \Phi_{2^{-n}}^{u+2^{-n}}(\xi_n^t(\frac{1}{2^n})), \Phi_{2^{-n}}^{u+2^{-n}}(\xi_n^u(\frac{1}{2^n}))\bigr)
\end{align}
\\
We exploit Condition C to get
\begin{align}\label{Eq3:Theorem:uniqueness1}
d\bigl( \Phi_{2^{-n}}^{t+2^{-n}}(\xi_n^t(\frac{1}{2^n})), \Phi_{2^{-n}}^{u+2^{-n}}(\xi_n^t(\frac{1}{2^n}))\bigr)
 \leq \frac{C}{2^n}|t-u|^\alpha
\end{align}
and Condition A gives
\begin{align}\label{Eq4:Theorem:uniqueness1}\nonumber
d \bigl( \Phi_{2^{-n}}^{u+2^{-n}}(\xi_n^t(\frac{1}{2^n})), \Phi_{2^{-n}}^{u+2^{-n}}(\xi_n^u(\frac{1}{2^n}))\bigr)
&\leq d\bigl( \xi_n^t(\frac{1}{2^n}), \xi_n^u(\frac{1}{2^n})\bigr) \bigl(1+\frac{K_A}{2^n}\bigr)\\
&= d_1\bigl(1+\frac{K_A}{2^n}\bigr)
\end{align}
\\
We combine(\ref{Eq2:Theorem:uniqueness1}), (\ref{Eq3:Theorem:uniqueness1}) and (\ref{Eq4:Theorem:uniqueness1})
\begin{align*}
d_2 &\leq d_1\bigl(1+\frac{K_A}{2^n}\bigr) + \frac{C}{2^n}|t-u|^\alpha\\
& \leq  \frac{C}{2^n}|t-u|^\alpha \bigl[ 1 + (1+\frac{K_A}{2^n}) \bigr]
\end{align*}
\\
Similarly, for all $k$ such that $\frac{k}{2^n}\leq c$, we have
\begin{align}\label{Eq5:Theorem:uniqueness1}\nonumber
d_k:&= d\bigl( \xi_n^t(\frac{k}{2^n}), \xi_n^u(\frac{k}{2^n})\bigr) \\\nonumber
&\leq \frac{C}{2^n}|t-u|^\alpha\bigl [ 1 + (1+\frac{K_A}{2^n})+ \cdots + (1+\frac{K_A}{2^n})^{k-1}\bigr]\\\nonumber
&\leq \frac{C|t-u|^\alpha}{2^n} \frac{[1+\frac{K_A}{2^n}]^k -1}{1+\frac{K_A}{2^n} -1}\\
&\leq \frac{C(e^{K_A/c}-1)}{K_A}|t-u|^\alpha
\end{align}
\\
This means, for all $s$ such that $s < c$
\begin{align}\label{Eq5:Theorem:uniqueness1}
d( \xi_n^t(s), \xi_n^u(s) \leq \tilde{C}|t-u|^\alpha
\end{align}
where $ \tilde{C}:=C(e^{K_A/c}-1)/K_A$. Since (\ref{Eq5:Theorem:uniqueness1}) is true for all $n$, we have
\begin{align}\label{Eq6:Theorem:uniqueness1}
d( \tilde{\sigma}_{b,t}(s), \tilde{\sigma}_{b,u}(s) \leq \tilde{C}|t-u|^\alpha
\end{align}

Now, let us estimate the first term in the right hand side of (\ref{Eq1:Theorem:uniqueness1}). We define
$$g(s):= e^{-K_As}d(\tilde{\sigma}_{a,t}(s),\tilde{\sigma}_{b,t}(s)) $$
For $h\geq 0$, we have
\begin{align*}
g(s+h)-g(s)&=e^{-K_A(s+h)}d(\tilde{\sigma}_{a,t}(s+h),\tilde{\sigma}_{b,t}(s+h))-e^{-K_As}d(\tilde{\sigma}_{a,t}(s),\tilde{\sigma}_{b,t}(s))\\
&\leq e^{-K_A(s+h)}(d(\Phi_h^{t+s}(\tilde{\sigma}_{a,t}(s)),\Phi_h^{t+s}(\tilde{\sigma}_{b,t}(s))+o(h))-e^{-K_As}d(\tilde{\sigma}_{a,t}(s),\tilde{\sigma}_{b,t}(s))\\
&\leq e^{-K_As}e^{-K_Ah}d(\tilde{\sigma}_{a,t}(s),\tilde{\sigma}_{b,t}(s))(1+K_Ah)-e^{-K_As}d(\tilde{\sigma}_{a,t}(s),\tilde{\sigma}_{b,t}(s))+o(h) \\
&= (e^{-K_Ah}(1+K_Ah)-1 )e^{-K_As}d(\tilde{\sigma}_{a,t}(s),\tilde{\sigma}_{b,t}(s))+o(h)\\
&= o(h)(g(s)+1)
\end{align*}
Hence,
$$D^+g(s):={\limsup_{h\rightarrow 0^+}}\bigl(\frac{g(s+h)-g(s)}{h}\bigr)\leq 0 $$
Consequently, $g(t)\leq g(0)$ equivalently
\begin{equation}\label{Eq7:Theorem:uniqueness1}
d(\tilde{\sigma}_{a,t}(s), \tilde{\sigma}_{b,t}(s))\leq e^{K_As}d(a,b), \quad {\rm for }\quad s\in[0,c)
\end{equation}
(\ref{Eq6:Theorem:uniqueness1}) and (\ref{Eq7:Theorem:uniqueness1}) together with (\ref{Eq1:Theorem:uniqueness1}) give (\ref{Check:uniqueness}) and
conclude the proof.
\end{proof}

\subsection{Generating flows}
As a corollary of Theorem \ref{Theorem1} and Theorem \ref{Theorem-uniqueness1}, if an arc field $\Phi$ satisfies Condition A,B and C then there is a
unique solution curve $\sigma_{a,t}: [t, t + c_{a,t}) \rightarrow X$ with initial position $a$ at time $t,$ for each $a\in X$
and $t\in [0,\infty)$. To guarantee that $c_{a,t}=\infty$ for all $a\in X$ and $t\in[0,\infty)$, we borrow the idea of \cite{BC}.

\begin{defn}
An arc field $\Phi$ is said to have {\it linear speed growth} if there is a point $x\in X$ positive constants $c_1(x)$ and $c_2(x)$ such that for all
$r>0$ and $t,l>0$
\begin{equation}\label{linear growth1}
\rho(x,t;r,l) \leq c_1(x)r + c_2(x)
\end{equation}
\end{defn}
\begin{thm}
Let $\Phi$ be an arc field which has linear speed growth. Suppose that at each point $a\in X$ and $t\in[0,\infty)$, $\Phi$ has a solution curve
$\sigma_{a,t}:[t,t + c_{a,t})\rightarrow X$ with initial position $a$ at time $t$. Then $c_{a,t}$ can be chosen to be $\infty$.
\end{thm}
\begin{proof}
Similar to Theorem 4.4 of \cite{BC}
\end{proof}

\section{Adding Flows}
\subsection{Sum of arc fields}
Let $\Phi$ and $\Psi$ be two arc fields satisfying Condition A,B and C. We impose a certain commutation
  law on $\Phi$ and $\Psi$.\\
{\bf Condition D:} There exist constants $0<\alpha<1$ and $C_d>0,$ and a function
$\Pi : X\times X \times [0,1] \rightarrow R$ such that for each $a\in X$ and $t\in [0,\infty)$,
there are constants $r_a>0,$ $\epsilon_a\in(0,1]$ and $T_t>t$ such that
 $\Pi $  is bounded from  above  on  $B(a,r_a)\times B(a,r_a)\times [0,\epsilon_a]$ and
\begin{equation}
d(\Phi_{2h}^{s+h}\circ\Psi_{2h}^s(b_1), \Psi_{2h}^{s+h}\circ\Phi_{2h}^s(b_2))\leq d(b_1,b_2)(1+h\Pi(b_1,b_2,h))+ C_d  h^{1+\alpha}
\end{equation}
for all $b_1,b_2\in B(a,r_a),$ $h\in[0,\epsilon_a]$ and $s,s+h \in[t,T_t]$.

\begin{rem}
Once we have fixed $a\in X$, $t\in [0,\infty)$ and fixed constants $r_a,\epsilon_a, T_t$ then the function $\Pi$ is bounded above.
We denote upper bound of $\Pi$ by $K_D$.
\end{rem}

\begin{defn}
For given $a\in X$ and $t\in[0,\infty),$ a solution curve of the
 sum of $\Phi$ and $\Psi$ with initial position $a$ at time $t$ is a map
$\sigma:[t,t+c)\rightarrow X$  such that $\sigma(t)=a$
and for each $s\in[t,t+c)$
\begin{equation}
\lim_{h\rightarrow 0}\frac{d(\sigma(s+2h),\Psi_{2h}^{s+h}\circ\Phi_{2h}^{s}(\sigma(s)))}{2h} =0
\end{equation}
\end{defn}

\begin{rem} If $\Phi$ and $\Psi$ satisfy Condition D, then we can check that
\begin{equation}
\lim_{h\rightarrow 0}\frac{d(\sigma(s+2h),\Psi_h^{s+h}\circ\Phi_h^{s}(\sigma(s)))}{2h} =
\lim_{h\rightarrow 0}\frac{d(\sigma(s+2h),\Phi_h^{s+h}\circ\Psi_h^{s}(\sigma(s)))}{2h}
\end{equation}
So, a solution curve of the  sum of $\Phi$ and $\Psi$ is a solution curve of the
 sum of $\Psi$ and $\Phi$.
\end{rem}

For a notational convenience in later computations, let us introduce new arc fields $\tilde{\Phi}$ and $\tilde{\Psi}$ which are defined by
$$\tilde{\Phi}_h^s(b):= \Phi_{2h}^s(b), \qquad  \tilde{\Psi}_h^s(b):= \Psi_{2h}^s(b)$$ for all $b\in X$ and $(h,s)\in [0,1]\times[0,\infty ).$
It is trivial to see if $\Phi$ and $\Psi$ satisfy Condition D then $\tilde{\Phi}$ and $\tilde{\Psi}$ satisfy the following condition\\
{\bf Condition D':} For all $b_1,b_2\in B(a,r_a),$ $h\in[0,\epsilon_a]$ and $s,s+h \in[t,T_t],$ we have
\begin{equation}\label{Condition:sum}
d(\tilde{\Phi}_h^{s+h}\circ\tilde{\Psi}_h^s(b_1), \tilde{\Psi}_h^{s+h}\circ \tilde{\Phi}_h^s(b_2))\leq d(b_1,b_2)(1+h {\Pi}(b_1,b_2,h))+ C_d  h^{1+\alpha}
\end{equation}
with same constants $r_a,$ $\epsilon_a,$ $T_t,$ $\alpha,$ $C_d$ and a function $\Pi$ as in Condition D.

It is also trivial to see that $\sigma:[t,t+c)\rightarrow X$ is a solution curve of the sum of $\Psi$ and $\Psi$ if and only if $\sigma$ satisfies
\begin{equation}\label{Eq-definition-sum}
\lim_{h\rightarrow 0}\frac{d(\sigma(s+2h),\tilde{\Psi}_h^{s+h}\circ \tilde{\Phi}_h^{s}(\sigma(s)))}{2h} =0
\end{equation}
for all $s\in[t,t+c).$

\begin{lem}\label{One-step:sum}
 For a  given $a\in X$ and $t\in[0,\infty)$, let arc fields $\Phi$ and $\Psi$ satisfy Condition A,B,C and D', and $r_a$, $\epsilon_a$ and $T_t$ be
 constants in those conditions. If $b_1,b_2\in B(a,r_a)$,
$4h\in[0,\epsilon_a]$ and $s,s+4h\in [t,T_t]$  then we have
$$d(\Psi_{h}^{s+3h}\circ\Phi_{h}^{s+2h}\circ\Psi_{h}^{s+h}\circ\Phi_{h}^{s}(b_1),
\Psi_{2h}^{s+2h}\circ\Phi_{2h}^{s}(b_2)) \leq d(b_1,b_2)(1+hK)^3 + C(h)$$
where  $K:= \max\{K_A,K_B,K_D \}$ and $C(h):=C_d h^{1+\alpha}(1+hK)+h\tilde{g}(h,h)[1+(1+hK)^2]$.
\end{lem}
\begin{proof}
\begin{figure}[h]
\centering\centerline{
\includegraphics[width=17cm]{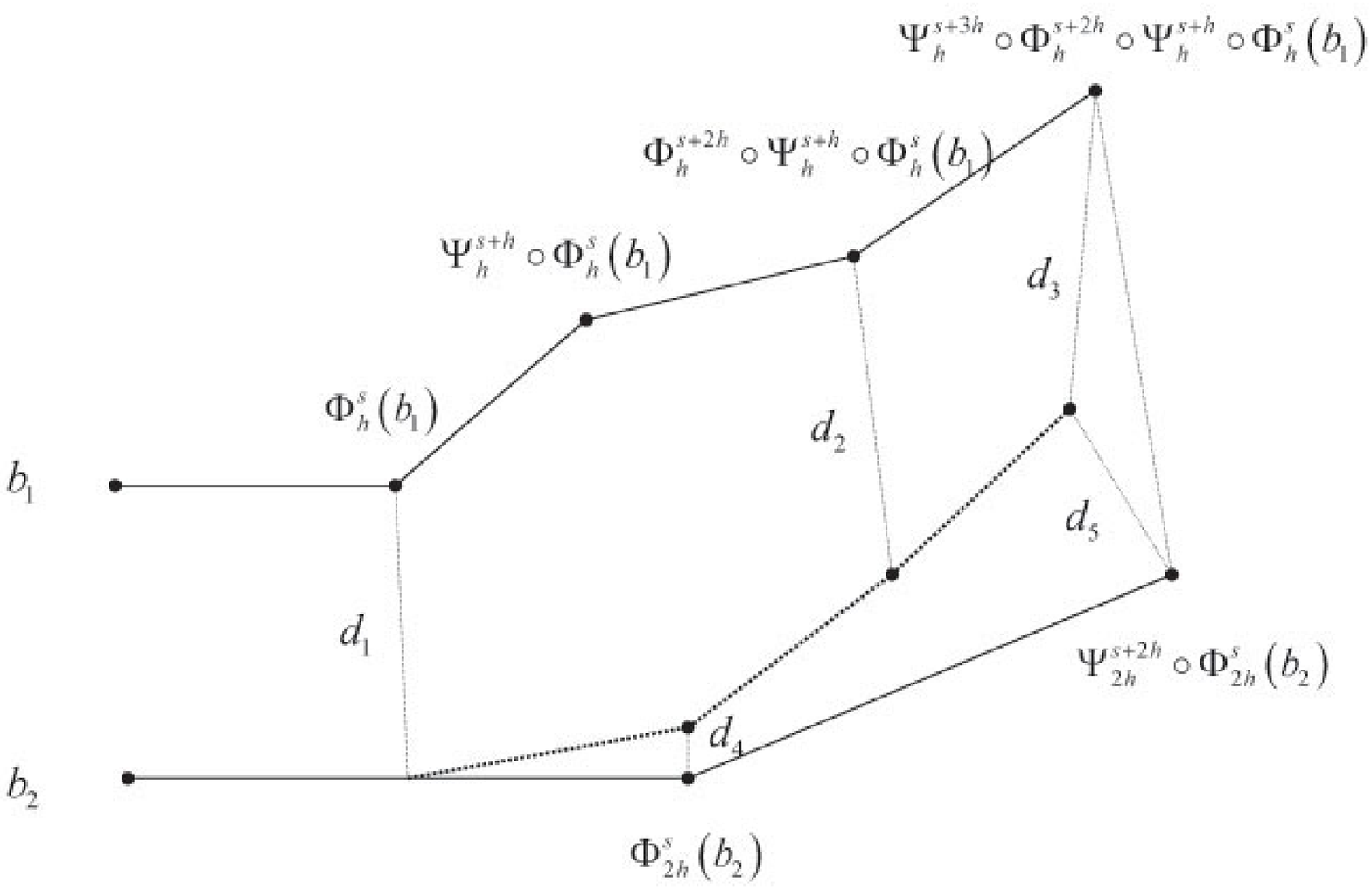}}
    \caption{
    }
\label{fig3_overload}
\end{figure}

1. We use Condition A to get,
\begin{align*}
d_1:= d(\Phi_h^s(b_1), \Phi_h^s(b_2))\leq d(b_1,b_2)(1+hK_A)
\end{align*}
2. Condition  A and D give,
\begin{align}\nonumber\label{Eq2:sum}
d_2:&= d(\Phi_h^{s+2h}\circ\Psi_h^{s+h}(\Phi_h^s(b_1)),\Psi_h^{s+2h} \circ  \Phi_h^{s+h}(\Phi_h^s(b_2)))\\\nonumber
&\leq d(\Phi_h^s(b_1),\Phi_h^s(b_2))(1+hK_D)+C_d h^{1+\alpha} \\
&\leq d(b_1,b_2)(1+hK_A)(1+hK_D)+C_d h^{1+\alpha}
\end{align}
3. We combine Condition A and (\ref{Eq2:sum}) to get
\begin{align}\label{Eq12:sum}\nonumber
d_3:&= d(\Psi_h^{s+3h}(\Phi_h^{s+2h}\circ\Psi_h^{s+h}\circ\Phi_h^s(b_1)), \Psi_h^{s+3h}(\Psi_h^{s+2h}\circ \Phi_h^{s+h}\circ\Phi_h^s(b_2)))\\\nonumber
&\leq d(\Phi_h^{s+2h}\circ\Psi_h^{s+h}\circ\Phi_h^s(b_1),\Psi_h^{s+2h}\circ \Phi_h^{s+h}\circ\Phi_h^s(b_2)))(1+hK_A) \\
&\leq [d(b_1,b_2)(1+hK_A)(1+hK_D)+C_d h^{1+\alpha}](1+hK_A)
\end{align}
4. Equation (\ref{Condition-D}) gives,
\begin{align}\label{Eq3:sum}
d_4:= d(\Phi_{2h}^{s}(b_2),\Phi_h^{s+h}\circ\Phi_h^s(b_2)))\leq h\tilde{g}(h,h)
\end{align}
5. We exploit Lemma \ref{One-step} to have
\begin{align*}
d(\Psi_h^{s+3h}\circ\Psi_h^{s+2h}(\Phi_h^{s+h}\circ\Phi_h^s(b_2))),\Psi_{2h}^{s+2h}(\Phi_{2h}^{s}(b_2)))
 \leq d(\Phi_h^{s+h}\circ\Phi_h^s(b_2)),\Phi_{2h}^{s}(b_2))(1+hK_A)^2+ h\tilde{g}(h,h)
\end{align*}
and this together with (\ref{Eq3:sum}) gives
\begin{align}\label{Eq13:sum}\nonumber
d_5:&= d(\Psi_h^{s+3h}\circ\Psi_h^{s+2h}(\Phi_h^{s+h}\circ\Phi_h^s(b_2))),\Psi_{2h}^{s+2h}(\Phi_{2h}^{s}(b_2)))\\
 &\leq h\tilde{g}(h,h)(1+hK_A)^2+ h\tilde{g}(h,h)
\end{align}
Finally, triangle inequality with (\ref{Eq12:sum}) and (\ref{Eq13:sum}) gives
\begin{align*}
& d(\Psi_{h}^{s+3h}\circ\Phi_{h}^{s+2h}\circ\Psi_{2h}^{s+h}\circ\Phi_{h}^{s}(b_1),
\Psi_{2h}^{s+2h}\circ\Phi_{2h}^{s}(b_2)) \\
&\leq d_3+d_5\\
 &\leq d(b_1,b_2)(1+hK_A)^2(1+hK_D) + C_d h^{1+\alpha}(1+hK_A)
+h\tilde{g}(h,h)[1+(1+hK_A)^2]
\end{align*}
\end{proof}

\begin{rem}\label{Remark:one-step:sum}
By exactly the  same argument as in Lemma \ref{One-step:sum}, we can show a  more general formula:
Let $h_1:=l_1+r$ and $h_2:= r+ l_2$, then we have
\begin{align}\label{Eq1:remark:one-step:sum}\nonumber
&d(\Psi_{l_2}^{s+h_1+r}\circ\Phi_{r}^{s+h_1}\circ\Psi_{r}^{s+l_1}\circ\Phi_{l_1}^s(b_1) , \Psi_{h_2}^{s+h_1}\circ\Phi_{h_1}^s(b_2))\\
&\qquad \quad \leq d(b_1,b_2)(1+l_1K_A)(1+rK_D)(1+l_2K_A) + C(l_1,r, l_2)
\end{align}
where $C(l_1,r,l_2):=r\tilde{g}(l_1,r)(1+l_2K_A)(1+rK_A) +l_2\tilde{g}(r,l_2)+ C_d r^{1+\alpha}(1+l_2K_A)$.

If we define $K:=\max \{K_A,K_D \}$ and $u:= \max \{l_1,r,l_2 \}$ then (\ref{Eq1:remark:one-step:sum}) can be simplified as
\begin{align}\label{Eq2:remark:one-step:sum}
d(\Psi_l^{s+h+r}\circ\Phi_{r}^{s+h}\circ\Psi_{r}^{s+l}\circ\Phi_l^s(b_1) , \Psi_h^{s+h}\circ\Phi_h^s(b_2)) \leq d(b_1,b_2)(1+uK)^3 + C(u)
\end{align}
where $C(u):= C_d u^{1+\alpha}(1+uK) + u\tilde{g}(u,u)[1 + (1+uK)^2]$.
\end{rem}


\subsection{Existence and Uniqueness of a solution for the sum of two arc fields}
\begin{thm}(Existence)\label{Theorem:existence:sum}
Let $\Phi,\Psi$ be arc fields satisfying Condition A,B,C and D. Then, for given $a\in X$ and $t\in[0,\infty)$, there is a solution curve $\sigma:[t,t+c)\rightarrow X$
of the sum of $\Phi $ and $\Psi $  with initial position $a$ at time $t.$
\end{thm}

\begin{proof}
Without loss of generality, we assume $\Phi$ and $\Psi$ satisfy Condition A,B,C and D', and find a curve $\sigma$ satisfying (\ref{Eq-definition-sum}).

For a positive integer n, we define the n-th discretized solution by
\begin{align*}
\xi_n(s):=
\left \{ \begin{matrix} \Phi_s^t(a) & 0\leq s\leq \frac{1}{2^n}\\
\\
\Psi_{s- 2^{-n}}^{t+ 2^{-n}}(\xi_n(\frac{1}{2^n})) & \frac{1}{2^n}\leq s \leq \frac{2}{2^n}\\
.\\
.\\
.\\
\Phi_{s-i\cdot 2^{-n}}^{t+i\cdot 2^{-n}}(\xi_n(\frac{i}{2^n})) & \frac{i}{2^n}\leq s\leq \frac{i+1}{2^n}\quad {\rm if} \quad i=2m\\
\\
\Psi_{s-i\cdot 2^{-n}}^{t+i\cdot 2^{-n}}(\xi_n(\frac{i}{2^n})) & \frac{i}{2^n}\leq s\leq \frac{i+1}{2^n}\quad {\rm if} \quad i=2m+1\\
.\\
.
\end{matrix} \right.
\end{align*}
\\
Suppose  $r,l>0$ are chosen so that $\rho(a,t;r,l)<\infty$. If $\rho(a,t;r,l)=0,$ then $\sigma(s):=a$ defines a solution curve. Thus we assume
$\rho(a,t;r,l)>0,$ and let
\begin{equation}
c:= \min\left \{\frac{r}{\rho(a,t;r,l)},l \right \}
\end{equation}
It is easy to see that we have $\xi_n(s)\in B(a,r)$ for $0\leq s< c$. This also implies $\{\xi_n\}_{n=1}^\infty$ is equi-Lipschitz with Lipschitz constant $\rho(a,t;r,l).$
Moreover, by choosing $r$ smaller if necessary, we may assume
there are constants $K_A$, $K_B$ and $\epsilon\in(0,1]$ such that $\Lambda(p,q,h)\leq K_A$ and $\Omega(p,l',h)\leq K_B$ for all
$p,q\in B(a,r)$ and $l,h'\in [0,\epsilon]$. We may also assume $\overline{B(a,r)}$ is a complete metric space.
\begin{figure}[h]
\centering\centerline{
\includegraphics[width=17cm]{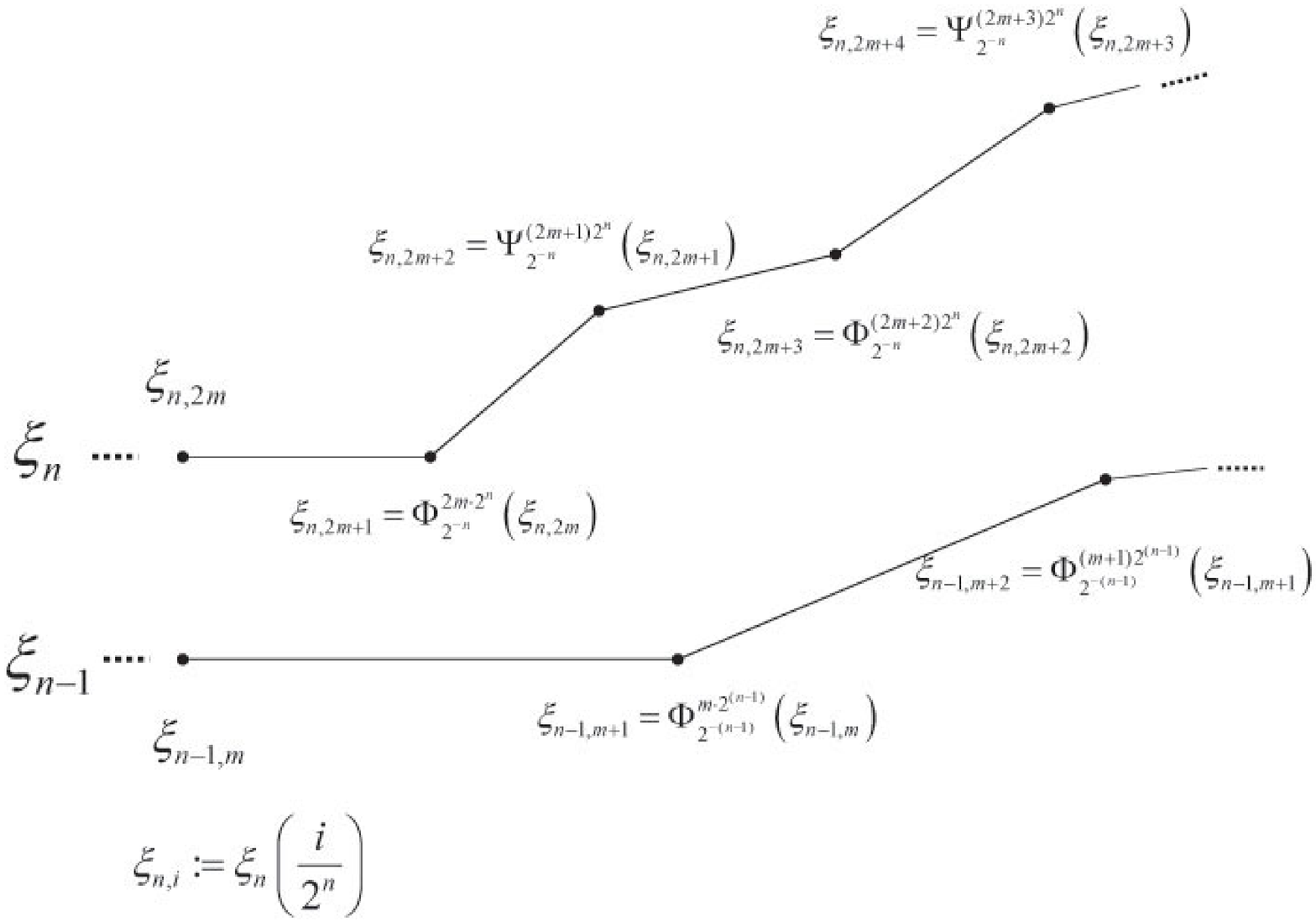}}
    \caption{
    }
\label{fig4_overload}
\end{figure}

Let us first estimate the uniform distance between $\xi_n$ and $\xi_{n-1}$.
We apply  Lemma \ref{One-step:sum} with $h:=1/2^n$, $s:=t + \frac{4i}{2^{n}}$, $b_1:=\xi_n(\frac{4i}{2^{n}})$
 and $b_2:=\xi_{n-1}(\frac{4i}{2^{n}})$  to have
\begin{align*}
d\bigl(\xi_n\bigl(\frac{4(i+1)}{2^{n}}\bigr),\xi_{n-1}\bigl(\frac{4(i+1)}{2^{n}}\bigr)\bigr)&= d\bigl(\Psi_{h}^{s+3h}\circ\Phi_{h}^{s+2h}\circ\Psi_{h}^{s+h}\circ\Phi_{h}^{s}
(\xi_n(\frac{4i}{2^{n}})),
\Psi_{2h}^{s+2h}\circ\Phi_{2h}^{s}(\xi_{n-1}(\frac{4i}{2^{n}}))\bigr)                   \\
& \leq d\bigl(\xi_n(\frac{4i}{2^{n}}),\xi_{n-1}(\frac{4i}{2^{n}})\bigr)\bigl(1+\frac{K}{2^n}\bigr)^3+C\bigl(\frac{1}{2^n}\bigr)
\end{align*}
where $K:=\max\{K_A,K_D \}$ and $ C(h)=C_d h^{1+\alpha}(1+hK) + h\tilde{g}(h,h)[1+(1+hK)^2] $.

In general, we have
\begin{align*}
d\bigl(\xi_n(\frac{4(i+1)}{2^{n}}),\xi_{n-1}(\frac{4(i+1)}{2^{n}})\bigr)& \leq d \bigl(\xi_n(\frac{4i}{2^{n}}),\xi_{n-1}(\frac{4i}{2^{n}})\bigr)
\bigl(1+\frac{K}{2^n}\bigr)^3+C\bigl(\frac{1}{2^n}\bigr)\\
& \leq d \bigl(\xi_n(\frac{4(i-1)}{2^{n}}),\xi_{n-1}(\frac{4(i-1)}{2^{n}})\bigr)\bigl(1+\frac{K}{2^n}\bigr)^{3\cdot 2} + C\bigl(\frac{1}{2^n}\bigr)
\bigl[1+(1+\frac{K}{2^n})^3\bigr]\\
&...\\
& \leq d(\xi_n(0),\xi_{n-1}(0))\bigl(1+\frac{K}{2^n}\bigr)^{3(i+1)} + C\bigl(\frac{1}{2^n}\bigr)
\bigl[1+\bigl(1+\frac{K}{2^n}\bigr)^3 + \cdots +\bigl(1+\frac{K}{2^n}\bigr)^{3i}\bigr]
\end{align*}
Since $\xi_n(0)=\xi_{n-1}(0) $, for all $i$ such that $4(i+1)/2^n\leq c$, we have
\begin{align}\label{Eq5:sum}\nonumber
d\bigl(\xi_n(\frac{4(i+1)}{2^{n}}),\xi_{n-1}(\frac{4(i+1)}{2^{n}})\bigr)
& \leq C(2^{-n}) \frac{(1+2^{-n}K)^{3i}-1}{(1+2^{-n}K)^3-1}\\\nonumber
& \leq C(2^{-n})\frac{e^{\frac{3c}{4}K}-1}{2^{-n}K[(1+2^{-n}K)^2 + (1+2^{-n}K)+1]}\\
& \leq \frac{C(2^{-n})}{2^{-n}}\frac{e^{\frac{3c}{4}K}-1}{3K}
\end{align}
Notice, for sufficiently small $h$
\begin{equation}\label{Property:C}
C(h)=C_d h^{1+\alpha}(1+hK) + h\tilde{g}(h,h)[1+(1+hK)^2] \leq 3h[C_d h^\alpha + \tilde{g}(h,h)]
\end{equation}
We assume $n$ is large enough in (\ref{Eq5:sum}). We combine that with (\ref{Property:C}) to show
\begin{align}\label{Eq4:sum}
d\bigl(\xi_n(\frac{4(i+1)}{2^{n}}),\xi_{n-1}(\frac{4(i+1)}{2^{n}})\bigr)
 \leq \bigl[C_d (\frac{1}{2^n})^\alpha + \tilde{g}(\frac{1}{2^n},\frac{1}{2^n})\bigr]A
\end{align}
where $A:= \frac{e^{\frac{3c}{4}K}-1}{K}$. We notice $A$ is independent of $n$.

So for any $s\in[0,c)$, let $i$ be an integer such that
$$\frac{4i}{2^n} \leq s < \frac{4(i+1)}{2^n} $$
then we have
\begin{align}\label{Eq14:sum}\nonumber
d(\xi_n(s),\xi_{n-1}(s))&\leq d\bigl(\xi_n(s),\xi_n(\frac{4i}{2^n})\bigr) + d\bigl(\xi_n(\frac{4i}{2^n}), \xi_{n-1}(\frac{4i}{2^n})\bigr)
+ d\bigl(\xi_{n-1}(\frac{4i}{2^n}),\xi_{n-1}(s)\bigr)\\
& \leq \frac{8}{2^n}\rho(a,t;r,l) + \bigl[C_d \bigl(\frac{1}{2^n}\bigr)^\alpha+ \tilde{g}\bigl(\frac{1}{2^n},\frac{1}{2^n}\bigr)\bigr]A
\end{align}
where we use the equi-Lipschitz property of $\xi_n$ and (\ref{Eq4:sum}).

Next we exploit (\ref{Eq14:sum}) to show $\xi_n$ is a Cauchy sequence in the uniform topology. For any $s\in [0,c)$, we have
\begin{align}\nonumber
d(\xi_n(s),\xi_{n+m}(s)) & \leq \sum_{j=0}^{m-1}  d(\xi_{n+j}(s),\xi_{n+j+1}(s))      \\\nonumber
& \leq \sum_{j=0}^{m-1}\bigl[\frac{8}{2^{n+j+1}}\rho(a,t;r,l) + \{C_d \bigl(\frac{1}{2^{n+j+1}}\bigr)^\alpha+
\tilde{g}\bigl(\frac{1}{2^{n+j+1}},\frac{1}{2^{n+j+1}}\bigr)\}A \bigr]\\\nonumber
&\leq \frac{8}{2^{n+1}}\rho(a,t;r,l)  + A\sum_{j=n+1}^{\infty}\bigl[C_d\bigl(\frac{1}{2^{n+j+1}}\bigr)^\alpha+
\tilde{g}\bigl(\frac{1}{2^{n+j+1}},\frac{1}{2^{n+j+1}}\bigr)\bigr]\\\nonumber
& \rightarrow 0 \quad {\rm as} \quad n\rightarrow 0
%
\end{align}
\\
Since $\xi_n(s)$ is in the complete space $\overline{B(a,r)}$, we know that $\xi_n$ converges uniformly. We
 define $\tilde{\sigma}:[0,c)\rightarrow X$ by the limit i.e
$$\tilde{\sigma}(s)=\lim_{n\rightarrow \infty}\xi_n(s) $$

Next, we are going to show that
\begin{equation}\label{Eq4:Theorem:existence:sum}
\lim_{h\rightarrow 0}\frac{d(\tilde{\sigma}(s+2h),\Psi_h^{t+s+h}\circ\Phi_h^{t+s}(\tilde{\sigma}(s)))}{2h}=0
\end{equation}
for all $s\in[0,c)$. Once we have (\ref{Eq4:Theorem:existence:sum}), we are done with the proof by defining a solution curve $\sigma:[t,t+c)\rightarrow X$ as
$$\sigma(s):=\tilde{\sigma}(s-t), \qquad {\rm for} \quad s\in[t,t+c) $$
To prove (\ref{Eq4:Theorem:existence:sum}), we choose an arbitrary $s\in[0,c)$. Let $\epsilon>0$ and $h>0$ be fixed such that $s+h<c$. From triangle inequality, we have

\begin{align}\label{Eq6:sum}\nonumber
\frac{d(\tilde{\sigma}(s+2h),\Psi_h^{t+s+h}\circ\Phi_h^{t+s}(\tilde{\sigma}(s)))}{h} & \leq \frac{d(\tilde{\sigma}(s+2h),\xi_n(s+2h))}{h}
 + \frac{d(\xi_n(s+2h),\Psi_h^{t+s+h}\circ\Phi_h^{t+s}(\xi_n(s)))}{h}\\
  & \qquad + \frac{d(\Psi_h^{t+s+h}\circ\Phi_h^{t+s}(\xi_n(s)),\Psi_h^{t+s+h}\circ\Phi_h^{t+s}(\tilde{\sigma}(s)))}{h}
\end{align}
Since $\xi_n$ converges uniformly, we can choose $n$ large enough so that

\begin{align}\label{Eq15:sum}
\frac{d(\tilde{\sigma}(s+2h),\xi_n(s+2h))}{h}+ \frac{d(\Psi_h^{t+s+h}\circ\Phi_h^{t+s}(\xi_n(s)),\Psi_h^{t+s+h}\circ\Phi_h^{t+s}(\tilde{\sigma}(s)))}{h}
  \leq \frac{\epsilon}{2}
\end{align}
We combine (\ref{Eq6:sum}) and (\ref{Eq15:sum}) to get
\begin{align}\label{Eq16:sum}
\frac{d(\tilde{\sigma}(s+2h),\Psi_h^{t+s+h}\circ\Phi_h^{t+s}(\tilde{\sigma}(s)))}{h} \leq  \frac{d(\xi_n(s+2h),\Psi_h^{t+s+h}\circ\Phi_h^{t+s}(\xi_n(s)))}{h}
+\frac{\epsilon}{2}
\end{align}
We need to estimate the second term of (\ref{Eq16:sum}).
\begin{figure}[h]
\centering\centerline{
\includegraphics[width=17cm]{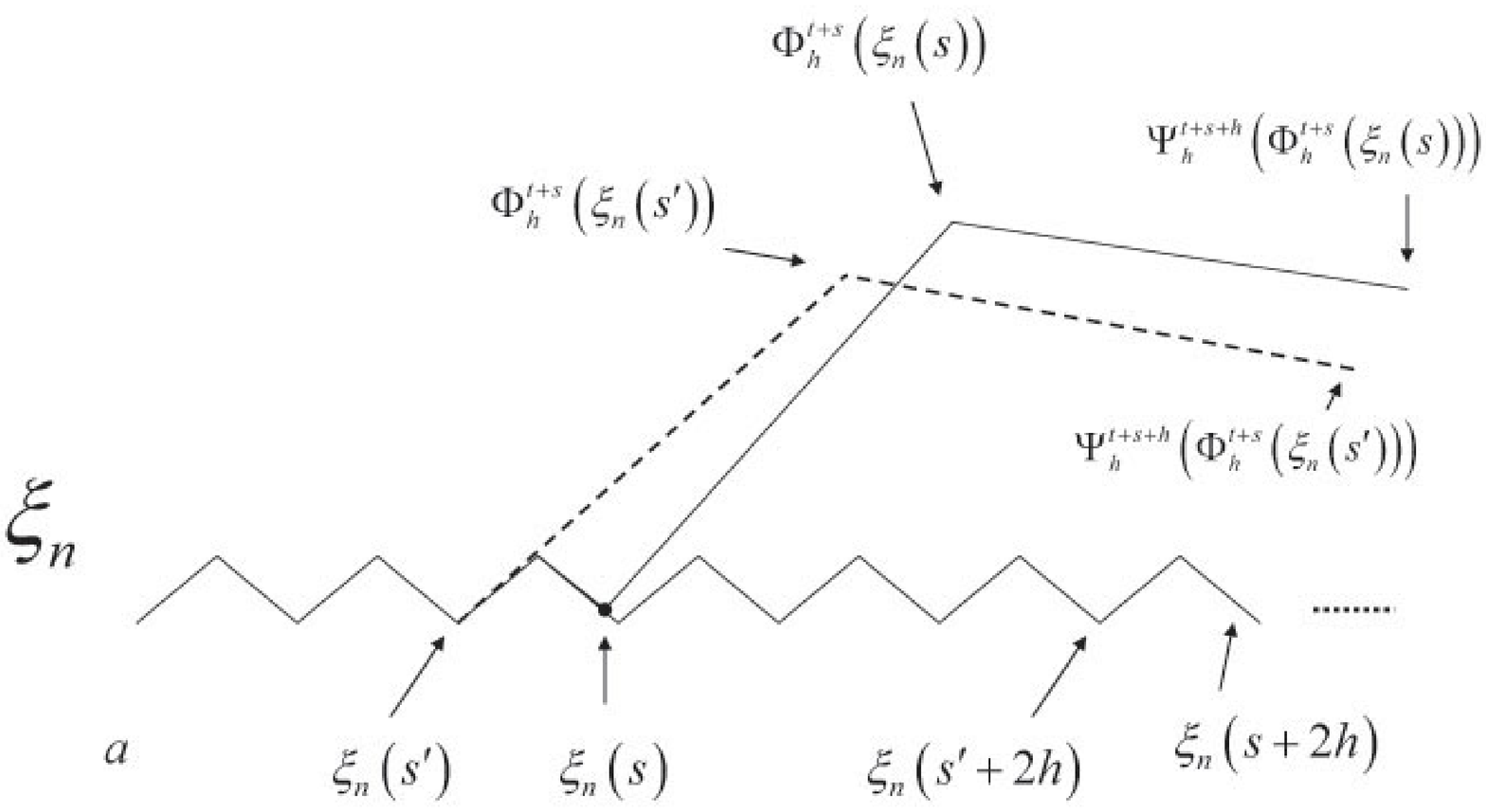}}
    \caption{
    }
\label{fig7_overload}
\end{figure}
%
Let $l$ be a nonnegative integer such that
$$\frac{2l}{2^n}\leq s < \frac{2(l+1)}{2^n}$$
and define $s':= \frac{2l}{2^n}$, then
\begin{align}\label{Eq18:sum}\nonumber
d(\xi_n(s+2h),\Psi_h^{t+s+h}\circ\Phi_h^{t+s}(\xi_n(s))) &\leq d(\xi_n(s+2h),\xi_n(s'+2h)) +
d(\xi_n(s'+2h),\Psi_h^{t+s'+h}\circ\Phi_h^{t+s'}(\xi_n(s')))\\\nonumber
 & \qquad + d(\Psi_h^{t+s'+h}\circ\Phi_h^{t+s'}(\xi_n(s')),\Psi_h^{t+s'+h}\circ\Phi_h^{t+s'}(\xi_n(s)))\\\nonumber
 & \qquad + d(\Psi_h^{t+s'+h}\circ\Phi_h^{t+s'}(\xi_n(s)),\Psi_h^{t+s'+h}\circ\Phi_h^{t+s}(\xi_n(s)))\\\nonumber
 & \qquad + d(\Psi_h^{t+s'+h}\circ\Phi_h^{t+s}(\xi_n(s)),\Psi_h^{t+s+h}\circ\Phi_h^{t+s}(\xi_n(s)))\\\nonumber
 & \leq d(\xi_n(s'+2h),\Psi_h^{t+s'+h}\circ\Phi_h^{t+s'}(\xi_n(s'))) \\\nonumber
& \qquad + \rho(a,t;r,l)(s-s') [1+(1+hK_A)^2] \\
& \qquad + Ch(s-s')^\alpha[1 + (1+hK_A)]
\end{align}
where we use Condition C and the Lipschitz continuity of $\xi_n$ in the second inequality.  For $n$ large enough, $s-s'$ is sufficiently small so that

\begin{align}\label{Eq17:sum}
 \frac{\rho(a,t;r,l)(s-s') [1+(1+hK_A)^2]  + Ch(s-s')^\alpha[1 + (1+hK_A)]}{h} \leq \frac{\epsilon}{2}
\end{align}
We combine (\ref{Eq18:sum}) and (\ref{Eq17:sum}) to get
\begin{align}\label{Eq11:sum}
\frac{d(\xi_n(s+2h),\Psi_h^{t+s+h}\circ\Phi_h^{t+s}(\xi_n(s)))}{h} &\leq \frac{ d(\xi_n(s'+2h),\Psi_h^{t+s'+h}\circ\Phi_h^{t+s'}(\xi_n(s')))}{h} +
\frac{\epsilon}{2}
\end{align}
\\
We need to estimate the second term of (\ref{Eq11:sum}). For the moment, let us assume $h=1/2^j$ for some $j\in N$ .
With this assumption, we can apply Lemma \ref{Discretization:sum} with $m:=n-j$ and $b:=\xi_n(s')$, and
get

\begin{align}\label{Eq3:Theorem:existence:sum}\nonumber
d(\Psi_{h}^{t+s'+h}\circ\Phi_{h}^{t+s'}(\xi_n(s')), \xi_n(s'+2h))&\leq C h\bigl [h^\alpha+\sum_{i=1}^{n-j}\tilde{g}\bigl(\frac{h}{2^i},\frac{h}{2^i}\bigr)\bigr]\\\nonumber
&=  \frac{C}{2^j}\bigl[\bigl(\frac{1}{2^j}\bigr)^\alpha+\sum_{i=1}^{n-j}\tilde{g}\bigl(\frac{1}{2^{i+j}},\frac{1}{2^{i+j}}\bigr)\bigr]\\
& \leq  \frac{C}{2^j}\bigl[\bigl(\frac{1}{2^j}\bigr)^\alpha+\sum_{i=j+1}^{\infty}\tilde{g}\bigl(\frac{1}{2^{i}},\frac{1}{2^{i}}\bigr)\bigr]
\end{align}
Notice (\ref{Eq3:Theorem:existence:sum}) is independent of $n$, i.e it holds uniformly for large $n\in N$. Now we combine (\ref{Eq6:sum}),
(\ref{Eq11:sum}) and
(\ref{Eq3:Theorem:existence:sum}), and let $h(= 1/2^j)\rightarrow 0$ then
 \begin{align}
\lim_{h\rightarrow 0}\frac{d(\tilde{\sigma}(t+2h),\Psi_h^{t+h}\circ\Phi_h^t(\tilde{\sigma}(t)))}{h}&\leq
\lim_{j\rightarrow \infty}C\bigl[\bigl({\frac{1}{2^j}}\bigr)^\alpha+\sum_{i=j+1}^{\infty}\tilde{g}\bigl(\frac{1}{2^{i}},\frac{1}{2^{i}}\bigr)\bigr] +
\epsilon =\epsilon
 \end{align}
This gives (\ref{Eq4:Theorem:existence:sum}).

For general $h$, let $k$ be an integer satisfying
\begin{equation}\label{Eq8:sum}
\frac{k}{2^n}\leq h< \frac{k+1}{2^n}
\end{equation}
 and define $h':=\frac{k}{2^n}$.
We exploit Lemma \ref{Check-solution:sum} to estimate the last term in (\ref{Eq16:sum}) and get
\begin{align}\label{Eq7:sum}\nonumber
\frac{d(\xi_n(s+2h),\Psi_h^{t+s+h}\circ\Phi_h^{t+s}(\xi_n(s)))}{h}
\leq &\frac{d(\xi_n(s+2h'),\Psi_{h'}^{t+s+h'}\circ\Phi_{h'}^{t+s}(\xi_n(s)))}{h'} \\\nonumber
 & + C\bigl [\frac{1}{k} + (\frac{1}{2^n})^\alpha\bigr ]\\
 \leq &\frac{d(\xi_n(s+2h'),\Psi_{h'}^{t+s+h'}\circ\Phi_{h'}^{t+s}(\xi_n(s)))}{h'} + \frac{\epsilon}{2}
\end{align}
where we assumed $n$ is large enough to get second inequality. We combine (\ref{Eq16:sum}) and (\ref{Eq7:sum})

\begin{align}\label{Eq7:Theorem:existence:sum}
\frac{d(\tilde{\sigma}(s+2h),\Psi_h^{t+s+h}\circ\Phi_h^{t+s}(\tilde{\sigma}(s)))}{h}\leq \frac{d(\Psi_{h'}^{t+s+h'}\circ\Phi_{h'}^{t+s}(\xi_n(s)), \xi_n(s+2h'))}{h'} +\epsilon
\end{align}

Let $m$ be a nonnegative integer such that $2^m \leq k < 2^{m+1}$ i.e
\begin{equation}\label{Eq19:sum}
\frac{2^m}{2^n} \leq h' < \frac{2^{m+1}}{2^n}
\end{equation}
\\
From Remark \ref{Remark:Discretization:sum} with $b:=\xi_n(s)$ and $u:= h'/k$, we have
\begin{align}\label{Eq5:Theorem:existence:sum}
d(\Psi_{h'}^{t+s+h'}\circ\Phi_{h'}^{t+s}(\xi_n(s)), \xi_n(s+2h'))\leq C2h'\bigl[(2h')^\alpha + \sum_{i=m+2}^\infty \tilde{g}\bigl(\frac{1}{2^i},\frac{1}{2^i}\bigr)\bigr]
\end{align}
\\
We combine (\ref{Eq7:Theorem:existence:sum}) and (\ref{Eq5:Theorem:existence:sum}) to get
\begin{align}\label{Eq20:sum}
\frac{d(\tilde{\sigma}(s+2h),\Psi_h^{t+s+h}\circ\Phi_h^{t+s}(\tilde{\sigma}(s)))}{h}\leq 2C\bigl[(2h')^\alpha +
\sum_{i=m+2}^\infty \tilde{g}\bigl(\frac{1}{2^i},\frac{1}{2^i}\bigr)\bigr] +\epsilon
\end{align}
\\
Notice that $h'$ converges to $h$ and $m $ increase to $\infty$ as $n\rightarrow \infty $. This implies, for all sufficiently large $n$, we have
\begin{align}\label{Eq21:sum}
 2C\bigl[(2h')^\alpha + \sum_{i=m+2}^\infty \tilde{g}\bigl(\frac{1}{2^i},\frac{1}{2^i}\bigr)\bigr]
 \leq 2C(2h)^\alpha + \epsilon
\end{align}
which is independent of $n$. We combine (\ref{Eq20:sum}) and (\ref{Eq21:sum}), and let $h$ converge to $0$ then we have

\begin{align}\nonumber
\lim_{h\rightarrow 0}\frac{d(\tilde{\sigma}(s+2h),\Psi_h^{t+s+h}\circ\Phi_h^{t+s}(\tilde{\sigma}(s)))}{h}\leq \lim_{h\rightarrow 0} 2C(2h)^\alpha + \epsilon = \epsilon
\end{align}

%
%
%
%
This gives (\ref{Eq4:Theorem:existence:sum}) and  concludes the proof.
\end{proof}

\begin{lem}\label{Check-solution:sum}
Let $\xi_n:[0,c)\rightarrow X$ be the n-th discretized solution constructed in Theorem \ref{Theorem:existence:sum} and $0<h< c$ be fixed.
If $k/2^n\leq h<(k+1)/2^n$ for some integer $k$ then there is a constant $C>0$ such that
 \begin{align*}
 \frac{d(\Psi_h^{u+h}\circ\Phi_h^u(\xi_n(s)), \xi_n(s+2h))}{h}&\leq \frac{d(\Psi_{h'}^{u+h'}\circ\Phi_{h'}^u(\xi_n(s)), \xi_n(s+2h'))}{h} +
 C \frac{|h-h'|}{h}\\
          & \qquad \qquad \qquad +C|h-h'|^\alpha
\end{align*}
Here, $h' = k/2^n$ and $C$ depends only on $c$ and $K_A.$
\end{lem}
\begin{proof}
 Triangle inequality gives
\begin{align}\label{Eq1:lemma:Check-solution:sum}\nonumber
& d(\Psi_h^{u+h}\circ\Phi_h^u(\xi_n(s)) , \xi_n(s+2h))\\\nonumber
&\leq d(\Psi_h^{u+h}\circ\Phi_h^u(\xi_n(s)), \Psi_{h'}^{u+h'}\circ\Phi_{h'}^u(\xi_n(s)))
 + d(\Psi_{h'}^{u+h'}\circ\Phi_{h'}^u(\xi_n(s)), \xi_n(s+2h')) \\
 & \quad + d(\xi_n(s+2h'), \xi_n(s+2h))
\end{align}
\\
Let us first estimate the last term of (\ref{Eq1:lemma:Check-solution:sum})
\begin{align}\label{Eq2:lemma:Check-solution:sum}
 d(\xi_n(s+2h'), \xi_n(s+2h))\leq 2\rho(a,t;r,l)|h-h'|
\end{align}
which comes from the Lipschitz continuity of $\xi_n$.

To estimate the first term of (\ref{Eq1:lemma:Check-solution:sum}), we exploit triangle inequality
\begin{align}\label{Eq1:integer:sum}\nonumber
d(\Psi_h^{u+h}\circ\Phi_h^u(\xi_n(s)), \Psi_{h'}^{u+h'}\circ\Phi_{h'}^u(\xi_n(s)))
&\leq d(\Psi_h^{u+h'}\circ\Phi_h^u(\xi_n(s)), \Psi_{h'}^{u+h'}\circ\Phi_{h'}^u(\xi_n(s)))\\
& \qquad +d(\Psi_h^{u+h}\circ\Phi_h^u(\xi_n(s)), \Psi_{h}^{u+h'}\circ\Phi_{h}^u(\xi_n(s)))
\end{align}
\\
Let us estimate the righthand side of (\ref{Eq1:integer:sum}) term by term.
For the first term, we use the Lipschitz continuity of $\Phi,\Psi$ and Condition A to get
\begin{align}\label{Eq2:integer:sum}\nonumber
d(\Psi_h^{u+h'}\circ\Phi_h^u(\xi_n(s)), \Psi_{h'}^{u+h'}\circ\Phi_{h'}^u(\xi_n(s)))
&\leq d(\Psi_h^{u+h'}\circ\Phi_h^u(\xi_n(s)), \Psi_{h'}^{u+h'}\circ\Phi_{h}^u(\xi_n(s)))\\\nonumber
                             & \qquad \qquad + d(\Psi_{h'}^{u+h'}\circ\Phi_h^u(\xi_n(s)), \Psi_{h'}^{u+h'}\circ\Phi_{h'}^u(\xi_n(s)))\\\nonumber
                             &\leq \rho(a,t;r,l) |h-h'| \\\nonumber
                             & \qquad \qquad + d(\Phi_h^u(\xi_n(s)),\Phi_{h'}^u(\xi_n(s)))(1+hK_A)\\
                             & \leq \rho(a,t;r,l) |h-h'|[1+ (1+hK_A)]
\end{align}
For the second term, we use Condition C and get
\begin{align}\label{Eq3:integer:sum}
d(\Psi_h^{u+h}\circ\Phi_h^u(\xi_n(s)), \Psi_{h}^{u+h'}\circ\Phi_{h}^u(\xi_n(s)))\leq Ch|h-h'|^\alpha
\end{align}
\\
Combine (\ref{Eq1:integer:sum}),(\ref{Eq2:integer:sum}) and (\ref{Eq3:integer:sum})
\begin{align}\label{Eq4:integer:sum}\nonumber
d(\Psi_h^{u+h}\circ\Phi_h^u(\xi_n(s)), \Psi_{h'}^{u+h'}\circ\Phi_{h'}^u(\xi_n(s)))
&\leq \rho(a,t;r,l) |h-h'|[1+ (1+hK_A)] \\
& \qquad+ Ch|h-h'|^\alpha
\end{align}
\\
Finally, (\ref{Eq1:lemma:Check-solution:sum}), (\ref{Eq2:lemma:Check-solution:sum}) and (\ref{Eq4:integer:sum}) give
\begin{align}\label{Eq4:integer:sum}\nonumber
d(\Psi_h^{u+h}\circ\Phi_h^u(\xi_n(s)) , \xi_n(s+2h))
&\leq  d(\Psi_{h'}^{u+h'}\circ\Phi_{h'}^u(\xi_n(s)), \xi_n(s+2h')) \\
 & \quad + \rho(a,t;r,l) |h-h'|[3+ (1+hK_A)]+ Ch|h-h'|^\alpha
\end{align}
Which concludes the proof.

\end{proof}

\begin{lem}\label{Discretization:sum}
Let $\Phi$ and $\Psi$ be arc fields satisfying Condition A,B,C and D'. For given $a\in X$ and $t\in[0,\infty)$, let $r_a,\epsilon_a$ and $T_t$ be
constants in those conditions. For $b\in B(a,r_a)$, $h\in[0,\epsilon_a]$ and $s,s+h\in[t,T_t] ,$ there exists a constant $C>0$ such that
\begin{equation*}
d(\Psi_h^{s+h}\circ\Phi_h^s(b),\Psi_{h/2^n}^{s+(2^{m+1}-1)h/2^n}\circ\Phi_{h/2^n}^{s+(2^{m+1}-2)h/2^n}\circ \cdots \circ\Psi_{h/2^n}^{s+h/2^n}
\circ\Phi_{h/2^n}^{s}(b) ) \leq  Ch\bigl[h^\alpha+\sum_{i=1}^m\tilde{g}\bigl(\frac{h}{2^i},\frac{h}{2^i}\bigr)\bigr]
\end{equation*}
where $m$ and $n$ are nonnegative integers satisfying $m\leq n.$ Here, $C$ depends only on $K$ and $T_t.$

\end{lem}

\begin{proof}
\begin{figure}[h]
\centering\centerline{
\includegraphics[width=17cm]{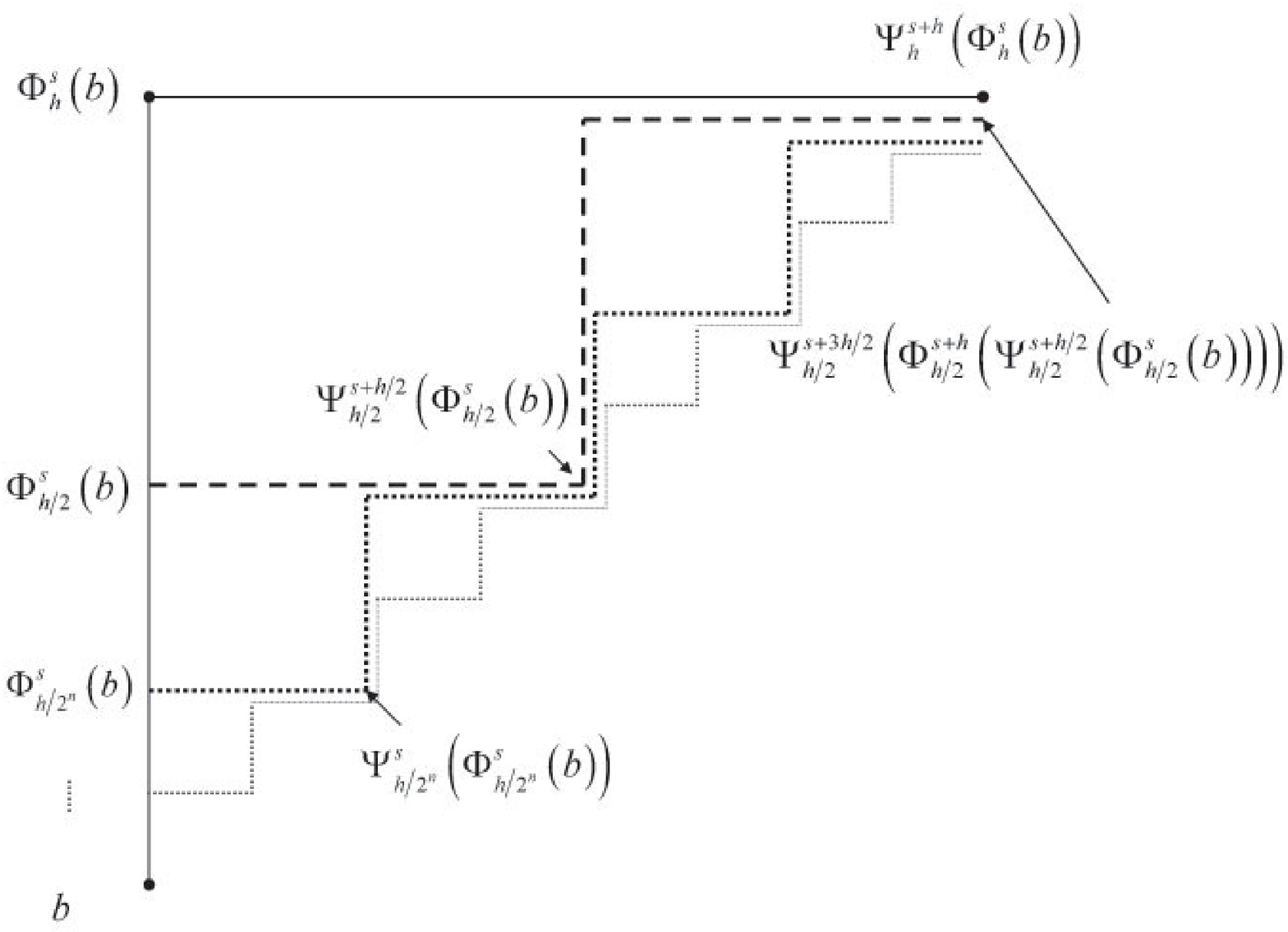}}
    \caption{
    }
\label{fig5_overload}
\end{figure}

We use Lemma \ref{One-step:sum},
\begin{align*}
d_1:=d(\Psi_h^{s+h}\circ\Phi_h^s(b),\Psi_{h/2}^{s+3h/2}\circ\Phi_{h/2}^{s+h}\circ \Psi_{h/2}^{s+h/2}\circ\Phi_{h/2}^{s}(b))
\leq C\bigl(\frac{h}{2}\bigr)
\end{align*}
where $C(h):= C_d h^{1+\alpha}(1+hK)+h\tilde{g}(h,h)[1+(1+hK)^2]$.
\begin{align}\label{Eq:d2:sum}\nonumber
d_2:&=d(\Psi_{h/2}^{s+3h/2}\circ\Phi_{h/2}^{s+h}\circ \Psi_{h/2}^{s+h/2}\circ\Phi_{h/2}^{s}(b),
\Psi_{h/2^2}^{s+7h/2^2}\circ\Phi_{h/2^2}^{s+6h/2^2}\circ \cdots \circ
\Psi_{h/2^2}^{s+h/2^2}\circ\Phi_{h/2^2}^{s}(b))\\
&\leq \tilde{d}_2\bigl(1+\frac{h}{2^2}K\bigr)^3 + C\bigl(\frac{h}{2^2}\bigr)
\end{align}
where $\tilde{d}_2:=d(\Psi_{h/2}^{s+h/2}\circ\Phi_{h/2}^{s}(b),\Psi_{h/2^2}^{s+3h/2^2}\circ\Phi_{h/2^2}^{s+2h/2^2}\circ
\Psi_{h/2^2}^{s+h/2^2}\circ\Phi_{h/2^2}^{s}(b) )$.\\

Again, by Lemma \ref{One-step:sum}
\begin{align}\label{Eq:td2:sum}
\tilde{d}_2\leq C\bigl(\frac{h}{2^2}\bigr)
\end{align}

We combine (\ref{Eq:d2:sum}) and (\ref{Eq:td2:sum}) to get

\begin{align}
d_2\leq C\bigl(\frac{h}{2^2}\bigr)\bigl[1+\bigl(1+\frac{h}{2^2}K\bigr)^3\bigr]
\end{align}

Similarly,
\begin{align*}
d_3:&= d(\Psi_{h/2^2}^{s+7h/2^2}\circ\Phi_{h/2^2}^{s+6h/2^2}\circ \cdots \circ\Psi_{h/2^2}^{s+h/2^2}\circ\Phi_{h/2^2}^{s}(b),
    \Psi_{h/2^3}^{s+15h/2^3}\circ\Phi_{h/2^3}^{s+14h/2^3}\circ \cdots \circ\Psi_{h/2^3}^{s+h/2^3}\circ\Phi_{h/2^3}^{s}(b) )      \\
 &\leq C\bigl(\frac{h}{2^3}\bigr)
\bigl[1+(1+\frac{h}{2^3}K)^3 +\bigl(1+\frac{h}{2^3}K\bigr)^{2\cdot 3} + \bigl(1+\frac{h}{2^3}K\bigr)^{3\cdot 3}\bigr]
\end{align*}

In general,
\begin{align*}
d_l&:= d(\Psi_{h/2^{l-1}}^{s+(2^l-1)h/2^{l-1}}\circ\Phi_{h/2^{l-1}}^{s+(2^l-2)h/2^{l-1}}\circ \cdots \circ\Psi_{h/2^{l-1}}^{s+h/2^{l-1}}\circ
\Phi_{h/2^{(l-1)}}^{s}(b)\\
& \qquad \qquad \qquad \qquad , \Psi_{h/2^l}^{s+(2^{l+1}-1)h/2^l}\circ\Phi_{h/2^l}^{s+(2^{l+1}-2)h/2^l}\circ \cdots \circ\Psi_{h/2^l}^{s+h/2^l}
\circ\Phi_{h/2^l}^{s}(b) )       \\
 &\leq C\bigl(\frac{h}{2^l}\bigr)
\bigl[1+(1+\frac{h}{2^l}K)^3 +\bigl(1+\frac{h}{2^l}K\bigr)^{2\cdot 3} + \cdots +\bigl(1+\frac{h}{2^l}K\bigr)^{(2l-1) 3}\bigr]\\
& \leq 2lC\bigl(\frac{h}{2^l}\bigr)\bigl(1+\frac{h}{2^l}K\bigr)^{(2l-1) 3}
\end{align*}
Notice $C(h)\leq 3[C_d h^{1+\alpha}+h\tilde{g}(h,h)]$ for small $h$. And there is a constant $C$ such that, for all $l$
$$\bigl(1+\frac{h}{2^l}K\bigr)^{(2l-1) 3}\leq C  $$
So we have
\begin{align*}
d_l\leq Cl\bigl[\bigl(\frac{h}{2^l} \bigr)^{1+\alpha}+\frac{h}{2^l}\tilde{g}\bigl(\frac{h}{2^l},\frac{h}{2^l}\bigr)\bigr]
\end{align*}
This implies
\begin{align*}
\sum_{i=1}^l d_i & \leq C\sum_{i=1}^l i\bigl [\bigl(\frac{h}{2^i} \bigr)^{1+\alpha}+\frac{h}{2^i}\tilde{g}\bigl(\frac{h}{2^i},\frac{h}{2^i}\bigr)\bigr]\\
&\leq Ch^{1+\alpha}\sum_{i=1}^l\frac{i}{2^i} + Ch\sum_{i=1}^l\tilde{g}\bigl( \frac{h}{2^i},\frac{h}{2^i}\bigr)\\
&\leq Ch\bigl[h^\alpha + \sum_{i=1}^l\tilde{g}\bigl( \frac{h}{2^i},\frac{h}{2^i}\bigr) \bigr]
\end{align*}
\end{proof}

\begin{rem}\label{Remark:Discretization:sum}
 Lemma \ref{Discretization:sum} says, if $k=2^m$ then
\begin{equation*}
d(\Psi_{ku}^{s+ku}\circ\Phi_{ku}^s(b),\Psi_{u}^{s+(2k-1)u}\circ\Phi_{u}^{s+(2k-2)u}\circ \cdots \circ\Psi_{u}^{s+u}
\circ\Phi_{u}^{s}(b) ) \leq  C(ku)[(ku)^\alpha+\sum_{i=1}^m\tilde{g}(\frac{ku}{2^i},\frac{ku}{2^i})]
\end{equation*}

In general, we can show that if $2^m \leq k < 2^{m+1}$ then
\begin{align}\label{Eq:Remark:Discretization:sum}\nonumber
&d(\Psi_{ku}^{s+ku}\circ\Phi_{ku}^s(b),\Psi_{u}^{s+(2k-1)u}\circ\Phi_{u}^{s+(2k-2)u}\circ \cdots \circ\Psi_{u}^{s+u}
\circ\Phi_{u}^{s}(b) ) \\
&\leq  C(2^{m+1}u)\bigl[(2^{m+1}u)^\alpha+\sum_{i=1}^{m+1}\tilde{g}\bigl(\frac{2^{m+1}u}{2^i},\frac{2^{m+1}u}{2^i}\bigr)\bigr]
\end{align}
\end{rem}

\begin{thm}(Uniqueness)\label{Theorem:uniqueness:sum}
Let $\sigma_{a,t}:[t,t+c)\rightarrow X$ be a solution curve of the
 sum of two  arc fields $\Phi$ and $\Psi$ with initial position $a$ at time $t$,
and let $\sigma_{b,u}:[u,u+c)\rightarrow X$ be a solution curve with initial position $b$ at time $u$.
 Then we have
$$d(\sigma_{a,t}(t+s), \sigma_{b,u}(u+s))\leq e^{K_As}d(a,b) + \tilde{C}|t-u|^\alpha , \quad {\rm for} \quad s\in[0,c) $$
where $\tilde{C}$ is a constant depending only on $c$ and $K_A$.
%
\end{thm}

\begin{proof}
Similar to the proof of Theorem \ref{Theorem-uniqueness1}.
\end{proof}

\subsection{Adding flows}
Like what we did for a  single arc field in section 2, we impose the linear speed growth condition on $\Phi$ and $\Psi$, and get the following theorem.

\begin{thm}
Let $\Phi$ and $\Psi$ be arc fields with linear speed growth. Suppose that at each point $a\in X$ and $t\in[0,\infty)$, there is a solution curve
of sum of $\Phi$ and $\Psi$
$$\sigma_{a,t}:[t, t + c_{a,t})\rightarrow X$$
 with initial position $a$ at time $t$. Then $c_{a,t}$ can be chosen to be $\infty$.
\end{thm}

Let $\mathcal{A}$ be the set of all time dependent arc fields which satisfy Condition A,B,C and linear speed growth condition.
 For each $\Phi \in \mathcal{A},$ solution curves of $\Phi$ generate a time dependent flow and we denote it by $\sigma_{\Phi}$.
 Likewise, we use notation $\sigma_{\Phi+\Psi}$ for the flow generated by  the
solution curves of the  sum of $\Phi$ and $\Psi$, when they satisfy Condition D.
 Notice that we have $\sigma_{\Phi+\Psi}=\sigma_{\Psi+\Phi}$ by symmetry.

 Let us define an equivalence relation $\sim$ in $\mathcal{A}$ as follows
$$ \Phi \sim \tilde{\Phi} \quad if \quad\sigma_{\Phi}=\sigma_{\tilde{\Phi}}$$
and we denote the equivalence class containing $\Phi$ by $[\Phi],$ i.e
$$[\Phi]:= \{\tilde{\Phi}\in\mathcal{A}:\Phi\sim \tilde{\Phi} \}$$
It is easy to see that $\sigma_{\Phi}$ can serve as an arc field and $\sigma_{\Phi}\in [\Phi].$

From the argument above,  there is an one to one correspondence between $\mathcal{A}/\sim$ and
the set of flows satisfying Condition A,B,C and the linear growth condition.
$$\mathcal{F}:=\{ \sigma_\Phi: \Phi\in\mathcal{A}\} \simeq \{[\Phi]: \Phi\in \mathcal{A} \} = \mathcal{A}/\sim $$

\begin{defn}\label{defn:sum:flows}
Let $\Phi,\Psi\in \mathcal{A}$ and suppose that $\Phi$ and $\Psi$ satisfy Condition D.
We define $\sigma_{\Phi} + \sigma_{\Psi}:= \sigma_{\Phi+\Psi}$ and call it the {\it sum }of two flows $\sigma_\Phi$ and $\sigma_\Psi.$
\end{defn}
\begin{lem}
The sum of  two flows $i.e$
$\sigma_\Phi +\sigma_\Psi$ is well defined.
\end{lem}
\begin{proof}
Let $\tilde{\Phi}\in[\Phi],$ $\tilde{\Psi}\in[\Psi]$ and suppose that  $\Phi$ and $\Psi$ satisfy Condition D.
We need to check
\begin{equation}\label{check:sum}
\sigma_{\Phi+\Psi}=\sigma_{\tilde{\Phi}+\tilde{\Psi}}
\end{equation}

To show (\ref{check:sum}), it is enough to prove
\begin{equation}\label{check:sum:eq1}
\lim_{h\rightarrow 0}\frac{d(\Phi_h^{s+h}\circ\Psi_h^s(b), \tilde{\Phi}_h^{s+h}\circ\tilde{\Psi}_h^s(b))}{h} =0
\end{equation}
for all $b\in X$ and $s\geq 0.$

Let $\sigma_\Psi:[s,\infty) \rightarrow X$ be the solution curve of arc field $\Psi$ with initial position $b$ at time $s.$ From Condition A,
we have
\begin{align}\nonumber\label{check:sum:eq2}
\lim_{h\rightarrow 0}\frac{d(\Phi_h^{s+h}(\Psi_h^s(b)),\Phi_h^{s+h}(\sigma_{\Psi}(s+h)))}{h}
&\leq \lim_{h\rightarrow 0}\frac{d(\Psi_h^s(b),\sigma_{\Psi}(s+h))(1+hK_A)}{h}\\
& =0
\end{align}
Similarly, we have
\begin{align}\nonumber\label{check:sum:eq3}
\lim_{h\rightarrow 0}\frac{d(\tilde{\Phi}_h^{s+h}(\tilde{\Psi}_h^s(b)),\tilde{\Phi}_h^{s+h}(\sigma_{\tilde{\Psi}}(s+h)))}{h}
&\leq \lim_{h\rightarrow 0}\frac{d(\tilde{\Psi}_h^s(b),\sigma_{\tilde{\Psi}}(s+h))(1+hK_A)}{h}\\
& =0
\end{align}
We combine (\ref{check:sum:eq2}), (\ref{check:sum:eq3}) and $\sigma_{\Psi}=\sigma_{\tilde{\Psi}}$, $[\tilde{\Phi}]=[\Phi] $ to get
\begin{align*}\nonumber
\lim_{h\rightarrow 0}\frac{d(\Phi_h^{s+h}\circ\Psi_h(b),\tilde{\Phi}_h^{s+h}\circ\tilde{\Psi}_h^s(b))}{h}
&\leq \lim_{h\rightarrow 0}\frac{d(\Phi_h^{s+h}(\sigma_{\Psi}(s+h)),\tilde{\Phi}_h^{s+h}(\sigma_{{\Psi}}(s+h)))}{h}\\
& =0
\end{align*}
which gives (\ref{check:sum:eq1}) and concludes proof.
\end{proof}
 Now, let us think about three flows $\sigma_\Phi, \sigma_\Psi$ and $ \sigma_\Theta.$ Suppose $\Phi$ and $\Psi$ satisfy Condition D, then
 the flow $\sigma_\Phi + \sigma_\Psi(=\sigma_{\Phi+\Psi}) $ is well defined. Furthermore, if $\Phi + \Psi$ and $\Theta$ satisfy Condition D, then
 $$\sigma_{(\Phi+\Psi)+\Theta}:=(\sigma_\Phi + \sigma_\Psi ) + \sigma_\Theta={\sigma}_{\Phi+\Psi} + \sigma_\Theta$$
 is also well defined. It is not hard to see that
 \begin{equation}\label{eq1:Sum:three}
 \lim_{h\rightarrow 0}\frac{d(\sigma_{(\Phi+\Psi)+\Theta}(s+3h),\Theta_{3h}^{s+2h} \circ \Psi_{3h}^{s+h}\circ\Phi_{3h}^{s}(\sigma(s)))}{3h} =0
 \end{equation}
 Similarly, if $\Psi$ and $\Theta$ satisfy Condition D, and $\Phi$ and $\Psi+\Theta$ satisfy Condition D then
$$\sigma_{\Phi+(\Psi+\Theta)}:=\sigma_\Phi + (\sigma_\Psi  + \sigma_\Theta)={\sigma}_{\Phi} + \sigma_{\Psi+\Theta}$$
is well defined. We also have
\begin{equation}\label{eq2:Sum:three}
\lim_{h\rightarrow 0}\frac{d(\sigma_{\Phi+(\Psi+\Theta)}(s+3h),\Theta_{3h}^{s+2h} \circ \Psi_{3h}^{s+h}\circ\Phi_{3h}^{s}(\sigma(s)))}{3h} =0
\end{equation}
By combining (\ref{eq1:Sum:three}) and (\ref{eq2:Sum:three}), we have the following associative law in the sum of flows.
\begin{cor}
Let $\Phi,\Psi, \Theta \in \mathcal{A},$ if they satisfy properties related to Condition D to make sense of sum of three flows, then we have
$$\sigma_\Phi + (\sigma_\Psi+\sigma_\Theta) = (\sigma_\Phi+\sigma_\Psi) + \sigma_\Theta $$
\end{cor}
%



Courant Institute, New York University, 251 Mercer street, New York, NY 10012, USA.

{\it E-mail address: hwakil@cims.nyu.edu}
\\

Courant Institute, New York University, 251 Mercer street, New York, NY 10012, USA.

{\it E-mail address: masmoudi@cims.nyu.edu}

\end{document}